\documentclass[11pt]{article}

\usepackage{cite}
\usepackage[hidelinks]{hyperref}
\usepackage[utf8]{inputenc}
\usepackage[shortlabels]{enumitem}
\usepackage{amsmath,amsfonts,amssymb,amsthm}
\usepackage{mathtools}
\usepackage{setspace}
\usepackage{verbatim}

\usepackage{geometry}
\usepackage{tikz}
\usetikzlibrary{cd}

\theoremstyle{plain}
\newtheorem{theorem}{Theorem}[section]
\newtheorem{lemma}[theorem]{Lemma}
\newtheorem{proposition}[theorem]{Proposition}
\newtheorem{corollary}[theorem]{Corollary}
\newtheorem*{conjecture}{Conjecture}

\theoremstyle{definition}
\newtheorem{definition}[theorem]{Definition}
\newtheorem{example}[theorem]{Example}

\theoremstyle{remark}
\newtheorem{remark}[theorem]{Remark}

\newcommand{\R}{\mathbb{R}}
\newcommand{\C}{\mathbb{C}}
\newcommand{\F}{\mathbb{F}}
\newcommand{\h}{\mathfrak{h}}
\newcommand{\g}{\mathfrak{g}}
\renewcommand{\t}{\mathfrak{t}}
\newcommand{\Aut}{\mathrm{Aut}}
\newcommand{\End}{\mathrm{End}}
\newcommand{\Diff}{\mathrm{Diff}}
\newcommand{\Hom}{\mathrm{Hom}}
\newcommand{\SO}{\mathrm{SO}}
\newcommand{\so}{\mathrm{so}}
\newcommand{\Ad}{\mathrm{Ad}}
\newcommand{\ad}{\mathrm{ad}}
\newcommand{\A}{\mathcal{A}}
\newcommand{\G}{\mathcal{G}}
\newcommand{\M}{\mathcal{M}}

\newcommand{\X}{\mathfrak{X}}
\newcommand{\tr}{\mathrm{tr}}
\newcommand{\gcal}{\text{\fontfamily{pzc}\selectfont g}}
\renewcommand{\d}{\mathrm{d}}
\renewcommand{\L}{\mathcal{L}}
\renewcommand{\i}{\mathrm{i}}

\title{Polysymplectic Reduction and the Moduli Space of Flat Connections}
\date{}
\author{Casey Blacker}

\begin{document}
\maketitle

\begin{abstract}
	 A polysymplectic structure is a vector-valued symplectic form, that is, a closed nondegenerate $2$-form with values in a vector space. We first outline the polysymplectic Hamiltonian formalism with coefficients in a vector space $V$, we then apply this framework to show that the moduli space of flat connections on a principal bundle over a compact manifold $M$ is a polysymplectic reduction of the space of all connections by the action of the gauge group with respect to a natural polysymplectic structure with values in an infinite dimensional Banach space. As a consequence, the moduli space inherits a canonical $H^2(M)$-valued presymplectic structure.
	 
	 Along the way, we establish various properties of polysymplectic manifolds. For example, a Darboux-type theorem asserts that every $V$-symplectic manifold locally symplectically embeds in a standard polysymplectic manifold $\Hom(TQ,V)$. We also show that both the Arnold conjecture and the well-known convexity properties of the classical moment map fail to hold in the polysymplectic setting.
\end{abstract}

\let\thefootnote\relax
\footnote{\textit{Date.} June 5, 2019}
\footnote{\textit{2010 Mathematics Subject Classification.} 53D05, 53D20, 53D30, 70S15.}
\footnote{\textit{Key words and phrases.} polysymplectic geometry, Hamiltonian reduction, moment maps, gauge theory.}

\tableofcontents

\section{Introduction}

Polysymplectic geometry was introduced by G\"unther \cite{Gunther87,Gunther87a} to provide a Hamiltonian counterpart to the Lagrangian formalism of classical field theory. A \emph{polysymplectic structure} on $M$ is a nondegenerate $2$-form $\omega\in\Omega^2(M,\R^k)$ for some $k\geq 1$. By decomposing $\omega$ as the direct sum $\oplus_i\omega_i$ of $k$ closed $2$-forms $\omega_i\in\Omega^2(M)$, this is seen to be equivalent to the earlier \emph{$k$-symplectic} formalism of Awane \cite{Awane84,Awane92}, in which a $k$-symplectic structure $(\omega_1,\ldots,\omega_k)$ on $M$ consists of $k$ closed $2$-forms $\omega_i$ with $\cap_i\ker\omega_i=0$. Indeed, the terms $k$-symplectic and polysymplectic appear nearly interchangeably throughout the literature, though in the latter case a condition that $k+1$ divide $\dim M$ is sometimes imposed \cite{LucasVilarino15}. We also note the independent introduction of this material as the \emph{$k$-almost cotangent} formalism in \cite{LeonMendezSalgado88,LeonMendezSalgado93}.

In addition to its applications in classical field theory \cite{MunteanuReySalgado04,Norris01,FulpLawsonNorris96,McLeanNorris00,ReyRoman-RoySalgado05}, polysymplectic geometry is also the subject of intrinsic mathematical interest \cite{Cappelletti-MontanoBlaga08,Blaga10,LeonVilarino13,SoldatenkovVerbitsky15}. In this regard, we note the independent work of Norris \cite{Norris93,Norris94,Norris97} on the canonical polysymplectic structure of the frame bundle on a smooth manifold, as well as the more recent appearance of \emph{k-symplectic Lie systems} \cite{LucasVilarino15,LucasTobolskiVilarino15}. Additionally, steps have been taken to relate polysymplectic geometry to the related study of multisymplectic geometry\cite{ForgerGomes13}, which similarly arises in field-theoretic contexts \cite{CarinenaCrampinIbort91,MarsdenPatrickShkoller98,Nissenbaum17}. In its relation to field quantization, the polysymplectic approach known as the \emph{precanonical formalism} \cite{Kanatchikov98,Kanatchikov98a,Kanatchikov16,Kanatchikov15,Kanatchikov04,Kanatchikov01}.

Our approach was developed with the aim of furnishing a setting in which a gauge-theoretic observation of Atiyah and Bott, described below, may be generalized to a broader class of manifolds. This has required a greater degree of attention to the space of coefficients $V$. As such, we define a \emph{$V$-symplectic structure} on $M$ to be a nondegenerate $2$-form $\omega\in\Omega^2(M,V)$ with values in the vector space $V$. This terminology will prove useful as the spaces of coefficients $V$ often arises naturally and without preferred identifications with $\R^n$.

In Section \ref{sec:vv_symplectic_vector_spaces}, we outline the local theory of $V$-symplectic vector spaces, which furnish the local models for the $V$-symplectic manifolds to follow. A \emph{$V$-symplectic vector space} consists of a vector space $U$ with a nondegenerate alternating bilinear form $\omega:U\times U\to V$.  Most symplectic constructions extend in a straightforward manner to the $V$-symplectic setting, though their properties often differ in important ways. For example, one highly consequential distinction between the linear symplectic and polysymplectic formalisms is that, while the double orthogonal $A^{\omega\omega}$ of a subspace $A\subseteq U$ satisfies $A^{\omega\omega}=A$ in the symplectic setting, we are only guaranteed to have $A^{\omega\omega}\supseteq A$ in the polysymplectic context. The polysymplectic orthogonal will be a primary object of study.

We illustrate the local theory with four characteristic examples:

\begin{enumerate}[A.]
	\item The $\R^N$-symplectic vector space $(U,\oplus_i\omega_i)$ consisting of an even-dimensional vector space $U$ and the sum of $N$ classical symplectic forms $\omega_i$.
	\item The $\g$-symplectic vector space $\big(\g,[\,,]\big)$ consisting of a centerless Lie algebra $\g$ with its Lie bracket $[\,,]$.
	\item The $\R^3$-symplectic vector space $(\R^3,\times)$, where $\times$ is the cross product.
	\item The $V$-symplectic vector space $\big(U\oplus\Hom(U,V),\omega\big)$, for any vector spaces $U$ and $V$ of positive dimension, where $\omega(u+\phi,u'+\phi')=\phi'(u)-\phi(u')$ extends the canonical symplectic structure on $T^*Q$.
\end{enumerate}

We will show the last example to be universal in the sense that every $V$-symplectic vector space $(U,\omega)$ naturally polysymplectically embeds in $U\oplus\Hom(U,V)$.

In Section \ref{sec:v-hamiltonian formalism} we present the $V$-symplectic counterpart to the Hamiltonian formalism. As with the theory of $V$-symplectic vector spaces, the constructions of the classical Hamiltonian formalism find a natural $V$-valued equivalents, thought often with a greater variability of behavior. For example, when $\dim V\geq 2$, it is no longer the case that every function $f:M\to V$ is Hamiltonian, or that the reduction $(M_0,\omega_0)$ of a $V$-Hamiltonian system $(M,\omega,G,\mu)$ is necessarily a $V$-symplectic manifold, as the reduced $2$-form $\omega_0$ may be degenerate. In addition, we will show that the Arnold conjecture and the convexity properties of the classical moment map do not obtain in the polysymplectic context.

We obtain seven examples of $V$-symplectic manifolds:

\begin{enumerate}[A.]
	\item The $\R^N$-symplectic manifold $(M,\oplus_i\omega_i)$ consisting of an even-dimensional manifold $M$ equipped with the fiberwise sum $\oplus_i\omega_i\in\Omega^2(M,\R^N)$ of a collection of $N$ classical symplectic structures $\omega_i$ on $M$.
	\item The $\g$-symplectic manifold $(G,-\d\theta)$ comprising a Lie group $G$ with discrete center $Z(G)$ with $\g$-symplectic potential $\theta\in\Omega^1(G,\g)$ the Maurer-Cartan form on $G$.
	\item The $\R^3$-symplectic manifold $(Q_B,\omega_B)$ comprising the configuration space $Q_B$ of a rigid body $B$ under rotations about a fixed point in space, with polysymplectic structure $\omega_B$ induced by the principal homogeneous action of $\SO(3)$.
	\item The $V$-symplectic manifold $\big(\Hom(TQ,V),-\d\theta\big)$ for a manifold $Q$ and vector space $V$ of positive dimension, where $\theta$ is the canonical $V$-symplectic potential on $\Hom(TQ,V)$.
	\item The $V$-symplectic manifold $(TQ,\omega_L)$ associated to a $V$-mechanical system $(Q,L)$ with configuration space $Q$ and Lagrangian $L:TQ\to V$. Here $\omega_L$ is the pullback of the canonical $V$-symplectic structure above by the fiber derivative $\F L:TQ\to\Hom(TQ,V)$.
	\item The $\Omega^2(M)/B^2(M)$-symplectic manifold $(\A,\omega)$ comprising the space $\A$ of connections on a principal bundle $P$ over a base space $M$, with a polysymplectic form $\omega$ to be defined.
	\item Under suitable conditions, the polysymplectic reduction $(\M(P),\omega_0)$ of $(\A,\omega)$ inherits the structure of a $H^2(M)$-symplectic manifold.
\end{enumerate}

The first four are global extensions of the linear examples above, while the final two comprise the central topic of this paper. We will show that every $V$-symplectic manifold $(M,\omega)$ locally polysymplectically embeds in $\Hom(TM,V)$. The embedding is global precisely when $\omega$ is exact. The space $\Hom(TM,V)$, occasionally termed the \emph{polymomentum phase space}, provided the initial motivation for the polysymplectic formalism through its connection with classical field theory \cite{Gunther87}.

It is interesting to compare our list of examples to Kirillov's \cite{Kirillov04} three sources of classical symplectic manifolds:

\begin{enumerate}[i.]
	\item Algebraic submanifolds of the complex projective space $\C P^N$.
	\item The coadjoint orbits $\mathcal{O}\subseteq\g^*$ of a Lie group $G$.
	\item The momentum phase space $T^*Q$ of a smooth manifold $Q$.
\end{enumerate}

The coadjoint orbits $\mathcal{O}$ and the phase space $T^*Q$ find $V$-symplectic counterparts in the polysymplectic manifolds $(G,-\d\theta)$ and $\Hom(TQ,V)$. Though we do not investigate them here, we note that under suitable conditions the orbits of the coadjoint action $\Ad^*:G\curvearrowright\Hom(\g,V)$ posses natural $V$-symplectic forms \cite{Gunther87,MarreroRoman-RoySalgado15} which more directly extend the case of the classical coadjoint orbits in $\g^*$. If $G$ is centerless, then $(G,-\d\theta)$ is polysymplectomorphic to the orbit through the identity map $1_\g\in\Hom(\g,\g)$. On the other hand, there does not appear to be a natural candidate for a polysymplectic equivalent of $\C P^N$.

In Section \ref{sec:gauge_theory_in_higher_dimensions} we apply the $V$-symplectic framework in the setting of gauge theory. Atiyah and Bott observed \cite{AtiyahBott83} that the space of flat connections on a principal bundle over a closed surface is the symplectic reduction of the space of all connections by the action of the gauge group. The primary aim of this paper is to employ the polysymplectic formalism to extend this result beyond the surface case. This is the content of Theorem \ref{thm:reduction_of_the_space_of_connections}.

It is interesting to note that in the surface case the moduli space $\M(P)$ arises in the context of Jones-Witten topological quantum field theory \cite{Atiyah88,Witten89}. Though we do not do so here, it would be interesting to investigate similar connections with the material in this paper.

In light of the varied conventions that appear in the literature, we briefly comment on our notation and terminology. In this paper, a symplectic form $\omega$ is locally the negative exterior derivative $-\d\theta$ of a local potential $\theta$, the characterization of a Hamiltonian vector field $X_f$ involves a negative interior product $\d f = -\iota_{X_f}\omega$, and the induced vector field $\underline\xi\in\X(M)$ associated to $\xi\in\g$ for a Lie group action $G\curvearrowright M$ is equal to $\frac{\d}{\d t}\,e^{t\xi}\cdot x\,|_{t=0}$ at $x\in M$. By \emph{smooth} we mean $C^\infty$. All spaces are understood to be smooth unless otherwise noted.

\subsection{Results}

Let us briefly indicate those results which we believe may be of particular interest.

One realization of the Darboux theorem states that every symplectic manifold $(M,\omega)$ is locally symplectomorphic to $T^*L$, where $L\subseteq M$ is a Lagrangian submanifold of $M$, in such a way that identifies the zero section of $T^*L$ with $L$. It is well-known that there is no such theorem which holds in general for the polysymplectic setting. However, we do obtain the following much weaker result.

\setcounter{section}{3}
\setcounter{theorem}{5}
\begin{theorem}
	Every $V$-symplectic manifold $(M,\omega)$ locally polysymplectically embeds in $\Hom(TM,V)$.
\end{theorem}

Recall the celebrated Arnold conjecture.

\begin{conjecture}[Arnold \cite{McDuffSalamon98}]
		A symplectomorphism that is generated by a time-dependent Hamiltonian vector field should have at least as many fixed points as a Morse function on the manifold must have critical points.
\end{conjecture}

In contrast to the symplectic case, a $V$-symplectic counterexample is readily furnished.

\setcounter{section}{3}
\setcounter{theorem}{20}
\begin{theorem}
	The Arnold conjecture fails in the $V$-symplectic setting.
\end{theorem}

The fundamental theorem of $V$-Hamiltonian systems is presented as follows.

\setcounter{section}{3}
\setcounter{theorem}{21}
\begin{theorem}[Vector-Valued Hamiltonian Reduction]
	Let $(M,\omega,G,\mu)$ be a $V$-Hamiltonian system and fix $\alpha\in\Hom(\g,V)$. If the stabilizer subgroup $G_\alpha$ of $\alpha$ under the coadjoint action is connected, and if $M_\alpha=\mu^{-1}(\alpha)/G_\alpha$ is smooth, then there is a unique $V$-valued $2$-form $\omega_\alpha\in\Omega^2(M_\alpha,V)$ such that
	\[
		\pi^*\omega_\alpha = i^*\omega,
	\]
	where $i:\mu^{-1}(\alpha)\hookrightarrow M$ is the inclusion and $\pi:\mu^{-1}(\alpha)\to M_\alpha$ is the projection. The form $\omega_\alpha$ is closed and is nondegenerate at $\pi x$ if and only if $\underline{\g_\alpha}_x = \underline\g_x^{\omega\omega}\cap\underline\g_x$.
\end{theorem}

It should be noted that this result has been obtained previously, in a slightly different form, in \cite{MarreroRoman-RoySalgado15}. We achieve the following result on the smooth structure of the reduced space.

\setcounter{section}{3}
\setcounter{theorem}{27}
\begin{theorem}
	Let $(M,\omega,G,\mu)$ be a $V$-Hamiltonian system with compact Lie group $G$, and suppose that $\alpha\in\Hom(\g,V)$ is a regular value of the moment map $\mu:M\to\Hom(\g,V)$. Then the reduced space $M_\alpha$ has at most orbifold singularities.
\end{theorem}

The main result of this paper concerns the reduction of the space of connections by the action of the gauge group.

\setcounter{section}{4}
\setcounter{theorem}{11}
\begin{theorem}
	Let $M$ be a compact manifold of dimension at least $3$, $G$ a compact matrix Lie group, $P$ a $G$-principal bundle on $M$ with connected gauge group $\G$, $\A$ the space of connections on $P$, and $k>\frac{1}{2}\dim M+1$ a fixed integer. Denote the the $W^{k,2}$ Sobolev completion of $\A$ by $\A_k$, and likewise for $\G$, $\gcal$, and $\Omega^*$, and write $\tilde\Omega^2(M)$ and $\tilde{B}^2(M)$ for the spaces of $C^1$ forms and coboundaries on $M$, respectively. The function
	\[
		\mu: \A_k \to	 \Hom\big(\gcal_{k+1},\tilde\Omega^2(M)/\tilde{B}^2(M)\big),
	\]
	given by
	\[
		\mu(A)(f) = \langle F_A\wedge f\rangle_{\tilde\Omega^2/\tilde{B}^2},	\hspace{.6cm} f\in\Omega_{k+1}^0(M,\ad P)\cong\gcal_{k+1},
	\]
	is a moment map for the action of $\G_{k+1}$ on $\A_k$ with respect to the polysymplectic structure $\omega\in\Omega^2\big(\A,\tilde\Omega^2(M)/\tilde{B}^2(M)\big)$ defined by
	\[
		\omega(\alpha,\beta) = \langle\alpha\wedge\beta\rangle_{\tilde\Omega^2/\tilde{B}^2},\hspace{.8cm}	\alpha,\beta\in \Omega_k^1(M,\ad P)\cong T_A\A_k.
	\]
	The reduced space $(\A_k)_0$ is the moduli space of flat connections $\M_k=F^{-1}(0)/\G_{k+1}$ on $P$. On the smooth points of $\M_k$, the reduced $2$-form $\omega_0$ takes values in the second cohomology $H^2(M)$.
\end{theorem}

\setcounter{section}{1}
\setcounter{theorem}{0}

\subsubsection*{Acknowledgements.}
	This work is based on the PhD thesis of the author \cite{Blacker18}. The author would like to thank his advisor, Xianzhe Dai, and acknowledge the fellowship support of the NSF RTG in Geometry and Topology at UCSB.

\section{The Local Theory}\label{sec:vv_symplectic_vector_spaces}

In this section, we introduce the key entity of $V$-symplectic geometry: the $V$-symplectic vector space. Our treatment begins with the basic definitions and culminates in a linear $V$-symplectic reduction theorem.

Throughout this exposition $U$ and $V$ will denote real vector spaces of differing roles. The space $U$ will represent the underlying space on which a vector-valued form $\omega$ is defined, while $V$ represents the space of coefficients. This notation is consistent with that of the following sections, where we consider manifolds modeled on $U$ and vector-valued forms with coefficients in $V$.

\subsection{$V$-Symplectic Vector Spaces}

We begin with the fundamental construction of this section.

\begin{definition}
	Let $U$ and $V$ be vector spaces. A \emph{$V$-symplectic structure} $\omega:U\times U\to V$ on $U$ is a $V$-valued alternating bilinear form which is nondegenerate in the sense that $\iota_u\omega = 0$ for $u\in U$ only if $u=0$. We call the pair $(U,\omega)$ a \emph{$V$-symplectic vector space}.
\end{definition}

Thus, a polysymplectic vector space $(U,\omega)$ is a $V$-symplectic vector space for some $V$. As in the symplectic case, there is a correspondence $\omega\mapsto\iota\omega$ between the $V$-symplectic structures on $U$ and the injective linear maps from $U$ to $\Hom(U,V)$.

\begin{example}
\begin{enumerate}[A.]
	\item Every classical symplectic vector space is an $\R$-symplectic vector space. More generally, for a family $(\omega_i)_{i\leq N}$ of symplectic structures on the even-dimensional vector space $U$, we define the $\R^N$-symplectic form $\oplus_i\omega_i:U\times U\to \R^N$ by
	\[
		\oplus_i\omega_i(u,u') = \oplus_i\big[\omega_i(u,u')\big].
	\]
	\item Recall that the \emph{center} $\mathfrak{z}$ of a Lie algebra $\big(\g,[\,,]\big)$ is the ideal
	\[
		\mathfrak{z} = \big\{ \xi\in \g\,\big|\, \ad_{\xi} = 0 \big\},
	\]
	and that $\g$ is said to be \emph{centerless} if $\mathfrak{z}=0$. For such a Lie algebra $\g$, the bracket $[\,,]$ is nondegenerate and thus constitutes a $\g$-symplectic structure on $\g$. Since the center $\mathfrak{z}$ is an abelian ideal of $\g$, this class of examples includes every semisimple Lie algebra.
	
	\item As a more concrete instance of part b., corresponding to $\g=\so(3)$, we consider the cross product, $\times$, as an $\R^3$-symplectic structure on $\R^3$. For nondegeneracy, we note that for any $X\in \R^3\backslash\{0\}$ and any orthogonal $Y\in\R^3\backslash\{0\}$, we have $\|X\times Y\| = \|X\|\cdot\|Y\|>0$.
	\item For vector spaces $U$ and $V$, of strictly positive dimension, the assignment
	\[
		\omega(u+\phi,u'+\phi') = \phi'(u) - \phi(u')
	\]
	defines a $V$-symplectic structure on $U\oplus\Hom(U,V)$. To see that $\omega$ is nondegenerate, let $u+\phi\in U\oplus\Hom(U,V)$ be any nonzero element, choose $u'\in U$ and $\phi'\in\Hom(U,V)$ so that precisely one of $\phi(u')$ and $\phi'(u)$ is nonzero, and observe that
	\[
		\omega(u+\phi,u'+\phi') = \phi'(u)-\phi(u') \neq 0.
	\]
	As a notational convenience, we will identify $U$ and $\Hom(U,V)$ with their images in $U\oplus\Hom(U,V)$.
\end{enumerate}
\end{example}


\begin{definition}
	Let $(U,\omega)$ and $(U',\omega')$ be $V$ and $V'$-symplectic vector spaces, respectively. A \emph{weak morphism} of polysymplectic vector spaces,
	\[
		f : (U,\omega) \to (U',\omega'),
	\]
	consists of a pair of linear maps
	\begin{align*}
		f_0: U	&\to		U'		\\
		f_1: V	&\to		V',
	\end{align*}
	such that $f_0^*\omega' = f_1\circ\omega$.
\end{definition}

We distinguish two classes of weak morphisms,
\begin{enumerate}[i.]
	\item If $f_1=1_V$ then we call $f$ a \emph{morphism} of $V$-symplectic vector spaces, and we identify $f$ with $f_0:U\to U'$.
	\item If $f_0=1_U$ then we call $f$ a \emph{morphism of coefficients}, and we identify $f$ with $f_1:V\to V'$. If $f_1:V\to V'$ is injective (resp.\ surjective) then we say that $f$ is an \emph{extension} (resp.\ \emph{reduction}) of coefficients.
\end{enumerate}

\begin{example}
\begin{enumerate}[A.]
	\item Let $(U,\omega)$ and $(U',\omega')$ be classical symplectic vector spaces. The space of classical linear symplectic maps from $(U,\omega)$ to $(U',\omega')$ is the space of morphisms of $\R$-symplectic vector spaces from $(U,\omega)$ to $(U'\omega')$. A map $f:(U,\omega)\to(U',\omega')$ is a weak morphism precisely when $f^*\omega' = \lambda\omega$ for some $\lambda\in\R$.
	
	The classical symplectic vector space $(U,\omega_i)$ is obtained by reducing the coefficients of $(U,\oplus_i \omega_i)$ from $\R^N$ to $\R$. The map $f:(M,\oplus_i\omega_i)\to(M',\oplus_i\omega_i')$ is a morphism if and only if it is a classical symplectic map from $(M,\omega_i)$ to $(M,\omega_i')$ for each $i\leq N$.
	\item Every Lie algebra morphism $f:\g\to\h$ is a weak morphism with $f_0=f_1=f$.
	\item Every rotation about the origin is a weak automorphism of $(\R^3,\times)$. The space of automorphisms is trivial.
	\item Every linear automorphism $\bar{f}:U\to U$ extends to a polysymplectic automorphism
	\begin{align*}
		f:	U\oplus\Hom(U,V)		&\to		U\oplus\Hom(U,V)				\\
			u+\phi				&\mapsto		\bar{f}u + \bar{f}_*\phi.
	\end{align*}
	 In particular, the $V$-symplectic structure on $U\oplus\Hom(U,V)$ is invariant under the induced action of $\Aut\,U$.
\end{enumerate}
\end{example}

\begin{proposition}
	If $(U,\bar\omega)$ is a $V$-symplectic vector space, then the map
	\begin{align*}
		i:	U	&\hookrightarrow		U\oplus\Hom(U,V)								\\
			u	&\mapsto				u-\textstyle\frac{1}{2}\iota_u\bar\omega
	\end{align*}
	is an inclusion of $V$-symplectic vector spaces. That is, the graph of $-\frac{1}{2}\iota\omega:U\to\Hom(U,V)$ is isomorphic to $(U,\omega)$
\end{proposition}

\begin{proof}
	Denote by $\omega$ the canonical $V$-symplectic form on $U\oplus\Hom(U,V)$. For any $u,u'\in U$, a direct computation yields
	\begin{align*}
		2i^*\omega(u,u')
			&=	\omega(u-\iota_u\bar\omega,u'-\iota_{u'}\bar\omega)		\\
			&=	-\bar\omega(u',u) + \bar\omega(u,u')						\\
			&=	2\bar\omega(u,u').
	\end{align*}
	The result follows as the injectivity of $i$ is clear.
\end{proof}

We will show in Theorem \ref{thm:phase_space_local_embedding} that this is the local manifestation of a global phenomenon.

Looking ahead to Section \ref{sec:gauge_theory_in_higher_dimensions}, we consider the following infinite-dimensional example.

\begin{example}
	Let $\Sigma$ be a closed orientable surface. In Subsection \ref{subsec:cohomology_as_symplectic_reduction} we consider the following three formal $V$-symplectic structures on the vector space $\Omega^1(\Sigma)$,
	\begin{center}
	\begin{tabular}{lcl}
		$\;\omega(\alpha,\beta)$			&$\in$&$V$						\\ \hline
										&								\\ [-.35cm]
		$\phantom{\int_\Sigma}\alpha\wedge\beta$					&&$\Omega^2(\Sigma)$				\\
		$\phantom{\int_\Sigma}\alpha\wedge\beta+B^2(\Sigma)$		&&$\Omega^2(\Sigma)/B^2(\Sigma)$	\\
		$\int_\Sigma\alpha\wedge\beta$	&&$\R$
	\end{tabular}
	\end{center}
	where $B^2(\Sigma)$ denotes the image of the exterior derivative $\d:\Omega^1(\Sigma)\to\Omega^2(\Sigma)$. The spaces $\big(\Omega^1(\Sigma),\omega\big)$ are related as follows.
	\begin{center}
	\begin{tikzcd}[every arrow/.append style={shift left}]
		\Omega^2(\Sigma,\R)\ar[d, "\text{\sl{reduction}}" right]															\\
		\Omega^2(\Sigma,\R)/B^2(\Sigma,\R)\ar[d, "\text{\sl{reduction}}" right]\ar[u, "\text{\sl{extension}}" left]		\\
		\R\ar[u, "\text{\sl{extension}}" left]
	\end{tikzcd}
	\end{center}
	The first space $\Omega^2(\Sigma)$ is the codomain of the natural polysymplectic structure on $\Omega^1(\Sigma)$ induced by the wedge product, while the third $\R$ corresponds to the classical symplectic structure defined by Atiyah and Bott \cite{AtiyahBott83}. It turns out that it is the intermediate space $\Omega^2(\Sigma)/B^2(\Sigma)$ that will prove most suitable for our purposes.
\end{example}

\begin{definition}
	A $V$-symplectic vector space $(U,\omega)$ is said to be \emph{irreducible} if every reduction of coefficients is an isomorphism of $V$-symplectic vector spaces.
\end{definition}

\begin{proposition}
	The $V$-symplectic space $U\oplus\Hom(U,V)$ is irreducible.
\end{proposition}

\begin{proof}
	Let $f:V\to V'$ be a linear map with $\dim V'<\dim V$, let $\phi_v\in\Hom(U,V)$ denote the function with constant value $v\in\ker f$, and observe that
	\[
		(f\circ\omega)(u+\phi,\phi_v) = f\big[\phi_v(u) - \phi(0)\big] = f(v) = 0
	\]
	for all $u+\phi\in U\oplus\Hom(U,V)$. Thus, $\phi_v\in\ker f\omega$ and we deduce that $f\omega$ is not a polysymplectic form. Therefore, $f$ is not a reduction of coefficients.
\end{proof}


\subsection{The Polysymplectic Orthogonal}

We now introduce the polysymplectic analogue of the symplectic orthogonal. We remark that much of this subsection follows readily from \cite{MarreroRoman-RoySalgado15,LeonVilarino13}.

\begin{definition}
	Let $A$ be a subspace of the $V$-symplectic vector space $(U,\omega)$. The \emph{polysymplectic orthogonal} of $A$ in $U$ is the subspace
	\[
		A^\omega = \{v\in U \,|\, \omega(A,v) = 0\}.
	\]
\end{definition}


We now collect various properties of the polysymplectic orthogonal that will prove useful in the development of the theory.

\begin{lemma}\label{lem:vv_symplectic_subspace_relations}
	Let $(U,\omega)$ be a $V$-symplectic vector space, with subspaces $A,A_i,B,B_i\subseteq U$ ($i\leq N$). Then,
	\begin{enumerate}[i.]
		\item $U^\omega = 0$ and $0^\omega = U$,
		\item If $A\subseteq B$, then $A^\omega\supseteq B^\omega$, \label{part:lem_vv_symplectic_subspace_relations_inclusion_reversal}
		\item $A\subseteq A^{\omega\omega}$, \label{part:lem_vv_symplectic_subspace_relations_double_orthogonal_containment}
		\item $A^\omega = A^{\omega\omega\omega}$,
		\item $\bigcap_i A_i^\omega = \big(\sum_i A_i\big)^\omega$,
		\item $\sum_i A_i^\omega \subseteq \big(\bigcap_i A_i\big)^\omega$.
	\end{enumerate}
\end{lemma}
\begin{proof}
	(i) and (ii) are immediate.
	\begin{enumerate}[i.]\setcounter{enumi}{2}
		\item We have
		\[
			u\in A \implies \;\; \forall u'\in A^\omega\,:\, \omega(u,u') = 0  \;\; \implies u\in A^{\omega\omega}.
		\]
		\item Apply (ii) to (iii) to obtain $A^\omega\supseteq (A^{\omega\omega})^\omega$ and note that (iii) alone provides $A^\omega \subseteq (A^\omega)^{\omega\omega}$.
		\item A direct computation yields
			\begin{align*}
				u\in \big({\textstyle\sum_i}A_i\big)^\omega
					&\iff	\forall u'\in {\textstyle\sum_i A_i}\,:\,\omega(u,u') = 0				\\
					&\iff	\forall i\leq N \,:\, \forall u_i\in A_i \,:\, \omega(u,u_i) = 0		\\
					&\iff	\forall i\leq N \,:\, u\in A_i^\omega									\\
					&\iff	u\in \textstyle\bigcap_i A_i^\omega.
			\end{align*}

		\item Applying (ii) to the inclusion $\bigcap_i A_i\subseteq A_j$, we deduce that $A_j^\omega \subseteq \big(\bigcap_iA_i\big)^\omega$ for all $j\leq N$, and thus $\sum_j A_j^\omega\subseteq \big(\bigcap_iA_i\big)^\omega$.
	\end{enumerate}
\end{proof}

\begin{example}
\begin{enumerate}[A.]
	\item Let $(U,\oplus_i\omega_i)$ be the $\R^N$-symplectic vector space as above. For $A\subseteq U$ and $u\in U$, it follows that
	\[
		u\subseteq A^{\oplus_i\omega_i} \iff\; \forall i\leq N \,:\, u\in A^{\omega_i},
	\]
	from which we conclude
	\[
		A^{\oplus_i\omega_i} = \textstyle\bigcap_i A^{\omega_i}.
	\]
	
	\item The polysymplectic orthogonal of a subspace $\mathfrak{a}\subseteq\g$ is the centralizer $\mathfrak{c}_\g(\mathfrak{a})=\big\{\xi\in\g\,\big|\,[\mathfrak{a},\xi]=0\big\}$ of $\mathfrak{a}$ in $\g$. If $\g$ is semisimple, then $\mathfrak{a}^{\omega\omega}=\mathfrak{a}$ if and only if $\mathfrak{a}\subseteq\g$ is an ideal.
	\item Let $e_1,e_2,e_3$ be the standard coordinate basis vectors of the $\R^3$-symplectic vector space $(\R^3,\times)$. Then
	\[
		\langle e_1\rangle^\omega = \{v\in\R^3 \,|\, v\times e_1 = 0\} = \langle e_1\rangle
	\]
	and
	\[
		\langle e_1,e_2\rangle^\omega = \{v\in\R^3 \,|\, v\times e_1 = v\times e_2 = 0\} = 0.
	\]
	Thus we have
	\begin{center}
	\begin{tabular}{c|c}
		$A$		&$A^\omega$		\\ \hline \vspace{-.3cm}
				&				\\
		$0$		&$\R^3$			\\
		$\ell$	&$\ell$			\\
		$w$		&$0$				\\
		$\R^3$	&$0$
	\end{tabular}
	\end{center}
	for any $1$-dimensional subspace $\ell$ and $2$-dimensional subspaces $w$.
	\item Let $A\subseteq U$ and $B\subseteq \Hom(U,V)$, define the subspace
	\[
		I(A) = \{\phi\,|\,\phi(A)=0\} \subseteq \Hom(U,V),
	\]
	and let $B^0\subseteq U$ be the annihilator of $B$. By noting that $A^\omega = U\oplus I(A)$ and $B^\omega = B^0\oplus\Hom(U,V)$, and invoking Lemma \ref{lem:vv_symplectic_subspace_relations}, we obtain
	\begin{align*}
		A^{\omega\omega} &= \big(U\oplus I(A)\big)^\omega	 = U \cap I(A)^\omega						\\
			&\hspace{.5cm}= U\cap \big(I(A)^0 \oplus \Hom(U,V)\big) = I(A)^0 = A.
	\end{align*}
	On the other hand, it is not generally true that $B^{\omega\omega}=B$.
\end{enumerate}
\end{example}


Parallel to the classical subspace designations, we apply the following terminology for the subspaces $A$ of a $V$-symplectic vector spaces $(U,\omega)$.

\begin{center}
\begin{tabular}{rl}
	term					&condition					\\ \hline\vspace{-.3cm}
							&							\\
	\emph{isotropic}		&$A\subseteq A^\omega$			\\
	\emph{coisotropic}		&$A^\omega\subseteq A$			\\
	\emph{Lagrangian}		&$A^\omega = A$				\\
	\emph{polysymplectic}	&$A^\omega \cap A= 0$		\\
\end{tabular}
\end{center}

\begin{example}
\begin{enumerate}[A.]
	\item The subspace $A\subseteq U$ is polysymplectic with respect to $\oplus_i\omega_i$ if it symplectic with respect to each $\omega_i$. This condition, however, is not necessary.
	\item The isotropic subspaces of $(\g,[\,,])$ are the precisely the abelian subalgebras. If $\g$ is semisimple, then the Lagrangian subspaces are precisely the Cartan subalgebras.
	\item The Lagrangian (resp.\ coisotropic) subspaces of $(\R,\times)$ are precisely the $1$-dimensional (resp.\ $2$ and $3$ dimensional) subspaces.
	\item The Lagrangian subspaces of $U\oplus\Hom(U,V)$ include $U$ and $\Hom(U,V)$, from which it follows that every subspace $A\subseteq U$ and $B\subseteq\Hom(U,V)$ is isotropic. Lemma \ref{lem:vv_symplectic_subspace_relations} yields
	\[
		(U\oplus B)^{\omega} = U\cap B^{\omega} = B^0,
	\]
	from which we deduce
	\[
		(U\oplus B)^\omega \cap (U\oplus B) = B^0\cap (U\oplus B) = B^0.
	\]
	Thus, $U\oplus B$ is polysymplectic if and only if the annihilator $B^0\subseteq U$ vanishes.
\end{enumerate}
\end{example}

As in the classical situation, Lagrangian subspaces cannot properly contain each other.

\begin{proposition}\label{prop:Lagrangian_subspace_containment}
	If $A\subseteq B$ are Lagrangian subspaces of $(U,\omega)$, then $A=B$.
\end{proposition}

\begin{proof}
	An application of Lemma \ref{lem:vv_symplectic_subspace_relations} yields $A = A^\omega \supseteq B^\omega = B$.
\end{proof}


\subsection{Reduction of $V$-Symplectic Vector Spaces}

\begin{theorem}\label{thm:linear_vv_symplectic_reduction}
	If $A$ is a subspace of the $V$-symplectic vector space $(U,\omega)$, then $\omega$ descends to a bilinear form $\omega_A$ on the quotient $A^\omega/(A\cap A^\omega)$ with kernel $(A^{\omega\omega}\cap A^\omega)/(A\cap A^\omega)$. In particular, $\omega_A$ is polysymplectic if $A^{\omega\omega}=A$.
\end{theorem}

\begin{proof}
	Let $u,u'\in A^\omega$, put $B=A\cap A^\omega$, and observe that
	\begin{align*}
		\omega(u+B,u'+B)
			&=	\omega(u,u') + \omega(u,B) + \omega(B,u') + \omega(B,B)		\\
			&=	\omega(u,u').
	\end{align*}
	Thus, $\omega$ descends to a well-defined form on $A^\omega/(A\cap A^\omega)$. If $u\in A^\omega$, then the condition that $\omega(u,A\cap A^\omega)=0$ obtains precisely when $u\in A^{\omega\omega}\cap A^\omega$. It follows that the kernel of the induced form on $A^\omega/(A\cap A^\omega)$ is equal to $(A^{\omega\omega}\cap A^\omega)/(A\cap A^\omega)$.
\end{proof}

The following corollary is immediate.

\begin{corollary}\label{cor:vv_symplectic_vs_reduction}
	If $A$ is an isotropic subspace of $(U,\omega)$, then $\omega$ descends to a bilinear form $\omega_A$ on $A^\omega/A$ with kernel $A^{\omega\omega}/A$. In particular, $\omega_A$ is polysymplectic if and only if $A^{\omega\omega}=A$.
\end{corollary}

We call $\big(A^\omega/(A\cap A^\omega),\omega_A\big)$ the \emph{reduction} of $(U,\omega)$ by $A\subseteq U$, with \emph{reduced space} $A^\omega/(A\cap A^\omega)$ and \emph{reduced form} $\omega_A$. In contrast to the classical case, the reduced form $\omega_A$ may be degenerate.

\begin{example}
\begin{enumerate}[A.]
	\item The reduction of $(U,\oplus_i\omega_i)$ by the subspace $A\subseteq U$ is the intersection $\cap_i\,A^{\omega_i}/(A\cap A^{\omega_i})$ of the reductions of $(U,\omega_i)$ by $A$ for each $i\leq N$.
	\item The reduction of $(\g,[\,,])$ by the subspace $\mathfrak{a}\subseteq\g$ is the quotient $\mathfrak{c}_\g(\mathfrak{a})/\mathfrak{z}(\mathfrak{a})$ of the centralizer $\mathfrak{c}_\g(\mathfrak{a})$ of $\mathfrak{a}$ by its center $\mathfrak{z}(\mathfrak{a}) = \mathfrak{a}\cap\mathfrak{c}_\g(\mathfrak{a})$. If $\g$ is semisimple and $\mathfrak{a}$ is an ideal, then the reduction is polysymplectic.
	\item The reduction of $(\R^3,\times)$ by any subspace is a point.
	\item The reduction of $U\oplus\Hom(U,V)$ by the subspace $A\subseteq U$ is the sum $U/A\oplus I(A)$. Since $A^{\omega\omega}=A$, Corollary \ref{cor:vv_symplectic_vs_reduction} ensures that the reduced form $\omega_A$ is always polysymplectic. Indeed, the reduction is naturally isomorphic to $U/A\oplus\Hom(U/A,V)$.
\end{enumerate}
\end{example}

\section{The $V$-Hamiltonian Formalism}\label{sec:v-hamiltonian formalism}

Having developed the theory of $V$-symplectic vector spaces, we turn our attention now to the global setting. Our aim is to arrive at a theory parallel to the classical Hamiltonian formalism. In particular, we would like to arrive at suitable definitions for the notions of Hamiltonian actions and symplectic reduction in the vector-valued context.


\subsection{$V$-Symplectic Manifolds}

The fundamental definition of $V$-symplectic geometry is as follows.

\begin{definition}\label{def:vv_symplectic_manifold}
	Fix a manifold $M$ and a vector space $V$. A \emph{$V$-symplectic structure} $\omega\in\Omega^2(M,V)$ on $M$ is a closed $2$-form which is nondegenerate in the sense that $\iota_X\omega=0$ only if $X=0$. We call the pair $(M,\omega)$ a \emph{$V$-symplectic manifold}.
\end{definition}

A polysymplectic manifold is a $V$-symplectic manifold for some vector space $V$.

If $(M,\omega)$ is \emph{exact}, that is, if $\omega=-\d\theta$ for some $\theta\in\Omega^1(M,V)$, then we call $\theta$ a \emph{$V$-symplectic potential} for $\omega$.

\begin{example}
\begin{enumerate}[A.]
	\item Let $M$ be a smooth and even-dimensional manifold and suppose that $(\omega_i)_{i\leq N}$ is a collection of symplectic forms on $M$. The map $\oplus_i\omega_i\in \Omega^2(M,\R^N)$, given by
	\[
		(\oplus_i\omega_i)(X,Y) = \oplus_i\big[\omega_i(X,Y)\big]
	\]
	is evidently an $\R^N$-symplectic form on $M$.
	
	\item Let $G$ be a Lie group with discrete center $Z(G)$ and denote by $\theta\in\Omega^1(G,\g)$ the Maurer-Cartan form, that is,
	\[
		\theta_g(X) = (\lambda_{g^{-1}})_*X\in \g,
	\]
	for $g\in G$ and $X\in T_gG$, where $\lambda$ is the left regular representation on $G$. Since the Maurer-Cartan identity asserts that
	\[
		-\d\theta_g(X,Y) = \big[\theta_g X,\theta_g Y\big],
	\]
	and since the center of $\g$ is trivial, it follows that $-\d\theta\in\Omega^2(G,\g)$ is nondegenerate and thus constitutes a $\g$-symplectic form on $G$.
	\item Let $Q_B$ be the configuration manifold of a rigid body $B$ in ambient $3$-space $S$ under rotations about a basepoint $O\in S$. We identify $Q_B$ with the space of pointed orientation-preserving isometries from $(S,O)$ to $(\R^3,0)$. The natural identification of the infinitesimal rotation $X\in T_qQ_B$ with the angular velocity vector $\theta_q(X)\in\R^3\cong_q S$ constitutes a polysymplectic potential $\theta\in\Omega^1(Q_B,\R^3)$ for $\omega_B=-\d\theta = \theta\times\theta$, where $\times$ is the cross product on $\R^3$.
	\item Define the \emph{canonical $1$-form} $\theta\in\Omega^1\big(\Hom(TQ,V),V\big)$ by
	\[
		\theta_\phi(X) = \phi(\pi_*X),
	\]
	where $\phi\in\Hom(TQ,V),\,X\in T_\phi\Hom(TQ,V)$, and $\pi:\Hom(TQ,V)\to Q$ is the projection map. By locally identifying the manifold $Q$ with the vector space $U$ on which it is modeled, it is readily shown that $-\d\theta$ induces the standard $V$-symplectic form on the vector space
	\[
		T_\phi\Hom(TQ,V) \cong U\oplus\Hom(U,V)
	\]
	for each $\phi\in\Hom(TQ,V)$. In particular, $-\d\theta$ is a $V$-symplectic structure on $\Hom(TQ,V)$.
\end{enumerate}
\end{example}

All of these spaces are regular in the sense that every two points have symplectomorphic neighborhoods. This property is similar to what in multisymplectic geometry is  known as \emph{flatness} \cite{RyvkinWurzbacher18}.

\begin{remark}
	If $L:TQ\to V$ is a smooth map with nonvanishing second variation along the fibers of $TQ$, then the fiber derivative $\F L:TQ\to \Hom(TQ,V)$ defines an immersion of $TQ$ in $\Hom(TQ,V)$. Moreover, the canonical $V$-symplectic form $\omega$ on $\Hom(TQ,V)$ pulls back to a $V$-symplectic form $\omega_L=\F L^*\omega$ on the velocity phase space $TQ$. Unlike the classical case, when $\dim V\geq 2$ the immersion $\F L$ is never an embedding and may even have compact image.
\end{remark}

There are two natural generalizations of the classical notion of a symplectic map.

\begin{definition}
	Let $(M,\omega)$ and $(M',\omega')$ be $V$ and $V'$-symplectic manifolds, respectively. A \emph{weak polysymplectic map},
	\[
		f : (M,\omega) \to (M',\omega')
	\]
	consists of a diffeomorphism and a linear transformation
	\begin{align*}
		f_0: M	&\to		M'		\\
		f_1: V	&\to		V',
	\end{align*}
	such that $f_0^*\omega' = f_1\circ\omega$. We call $f$ a \emph{polysymplectic map} when $V=V'$ and $f_1=1_V$, and we call $f$ a \emph{morphism of coefficients} when $M=M'$ and $f_0=1_M$. A morphism of coefficients $f$ is said to be an \emph{extension (resp.\ reduction) of coefficients} if $f_1$ is injective (resp.\ surjective).
\end{definition}

When there is no room for confusion, we will frequently identify $f$ with $f_0$ or $f_1$.

\begin{example}
\begin{enumerate}[A.]
	\item The diffeomorphism $f:M\to M'$ is a polysymplectic map from $(M,\oplus_{i\leq N}\omega_i)$ to $(M',\oplus_{i\leq N}\omega'_i)$ if and only if it is a classical symplectic map from $(M,\omega_i)$ to $(M',\omega'_i)$ for each $i\leq N$.
	\item If $G$ and $G'$ are Lie groups with discrete centers, and if $f:G\to G'$ is any homomorphism, then $f$ is a weak polysymplectic map from $(G,-\d\theta)$ to $(G,-\d\theta')$.
	
	Fix $g\in G$ and let $\lambda_g:G\to G$ denote the left multiplication by $g$. For any $h\in G$ and $X\in T_hG$, we have
	\[
		\lambda_g^*\theta_{gh}(X) = (\lambda_{(gh)^{-1}})_*(\lambda_g)_*X = \theta_h(X),
	\]
	from which we deduce $\lambda_g^*\theta = \theta$. We conclude that $\lambda_g^*\d\theta = \d\theta$, and thus $\lambda_g$ is a polysymplectomorphism of $(G,-\d\theta)$. This establishes the regularity of $(G,-\d\theta)$.
	\item If $B$ and $B'$ are two rigid bodies in $S$ with basepoints $O$ and $O'$, respectively, then the polysymplectic maps from $Q_B$ to $Q_{B'}$ are precisely the pointed orientation-preserving isometries from $(S,O)$ to $(\mathbb{A}^3,O')$. Additionally, the action of $SO(3)$ on $Q_B$ establishes a weak polysymplectomorphism from $(Q_B,\omega_B)$ to $(\SO(3),-\d\theta)$ which is natural up to the choice of reference configuration $q_0\in Q_B$.
	\item Any diffeomorphism $\bar{f}:Q\to Q$ extends to a polysymplectomorphism $f:\Hom(TQ,V)\to\Hom(TQ,V)$. In particular, the $V$-symplectic structure of $\Hom(TQ,V)$ is preserved by the action of $\Diff\,Q$.
	
	Polysymplectic spaces that are locally isomorphic to $\Hom(TQ,V)$ for some $V$ are referred to in the literature as \emph{standard} \cite{Gunther87,MarreroRoman-RoySalgado15}.
\end{enumerate}
\end{example}

Recall that the classical Darboux theorem asserts that every symplectic manifold $(M,\omega)$ is locally isomorphic to a cotangent bundle $T^*Q$. In the $V$-symplectic setting, we obtain a weaker result.

\begin{theorem}\label{thm:phase_space_local_embedding}
	Every $V$-symplectic manifold $(M,\bar\omega)$ locally polysymplectically embeds in $\Hom(TM,V)$.
\end{theorem}

\begin{proof}
	Let $O\subseteq M$ be any open set on which $\omega$ is exact, choose $\bar\theta\in\Omega^1(O,V)$ so that $\omega|_O=-\d\bar\theta$, and observe that the section
	\begin{align*}
		\bar\theta:	O	&\to			\Hom(TM,V)		\\
					x	&\mapsto		\;\bar\theta_x
	\end{align*}
	is a smooth embedding. For any $x\in O$ and $X\in T_xM$, the image $\bar\theta_*X$ is tangent to $\Hom(TM,V)$ at $\bar\theta_x$ so that
	\[
		\theta(\bar\theta_*X) = \bar\theta_x(\pi_*\bar\theta_*X) = \bar\theta_x(X),
	\]
	where we have used the fact that $\bar\theta$ is a section of $\pi:\Hom(TM,V)\to M$. We conclude that the embedding $\theta$ is polysymplectic.
\end{proof}


Let us turn briefly to describe certain special classes of submanifolds.

\begin{definition}
	Fix a $V$-symplectic manifold $(M,\omega)$. A smooth submanifold $N\subseteq M$ is said to be \emph{polysymplectic} (resp.\ \emph{isotropic}, \emph{coisotropic}, \emph{Lagrangian}) when the subspace $T_xN\subseteq T_xM$ is polysymplectic (resp.\ isotropic, coisotropic, Lagrangian) at every point $x\in N$.
\end{definition}

Equivalently, $N\subseteq M$ is polysymplectic (resp.\ isotropic) when the restriction of $\omega$ to $N$ is a $V$-symplectic structure (resp.\ the zero form) on $N$.

\begin{example}
\begin{enumerate}[A.]
	\item The submanifold $N\subseteq M$ is polysymplectic with respect to $\oplus_i\omega_i$ if it symplectic with respect to each $\omega_i$. However, this is only a sufficient condition.
	\item If $T\subseteq G$ is a maximal torus, then $T$ is a Lagrangian submanifold of $(G,-\d\theta)$.
	\item The Lagrangian (resp.\ coisotropic) submanifolds of $Q_B$ are precisely the $1$-dimensional (resp.\ $2$ and $3$-dimensional) submanifolds.
	\item The fibers of $\Hom(TQ,V)$ are Lagrangian submanifolds.
\end{enumerate}
\end{example}

\begin{proposition}
	Suppose that $(M,\omega)$ is a $V$-symplectic manifold and that $\omega'\in\Omega^2(M,V')$ is an extension of coefficients of $\omega\in\Omega^2(M,V)$. If $N\subseteq M$ is polysymplectic (resp.\ coisotropic) with respect to $\omega$, then $N$ is polysymplectic (resp.\ coisotropic) with respect to $\omega'$.
\end{proposition}

\begin{proof}
	Fix $x\in N$. Since $\omega'$ refines $\omega$, it follows that
	\[
		T_xN^{\omega'} \subseteq T_xN^\omega.
	\]
	Consequently, if $N$ is polysymplectic with respect to $\omega$, then
	\[
		T_xN \cap T_xN^{\omega'} \subseteq T_xN\cap T_xN^\omega = 0,
	\]
	and $N$ is polysymplectic with respect to $\omega'$. If $N$ is coisotropic with respect to $\omega$, then
	\[
		T_xN^{\omega'} \subseteq T_xN^\omega \subseteq T_xN,
	\]
	and thus $N$ is coisotropic with respect to $\omega$.
\end{proof}

\begin{proposition}
	If $N$ is a Lagrangian submanifold of a $V$-symplectic manifold $(M,\omega)$, then $N$ is not contained in any Lagrangian manifold of strictly greater dimension.
\end{proposition}

\begin{proof}
	This is an immediate consequence of Proposition \ref{prop:Lagrangian_subspace_containment}.
\end{proof}


\subsection{Polysymplectic and Hamiltonian Actions}

We continue the parallel development with the classical theory with the introduction of polysymplectic and Hamiltonian actions.

\begin{definition}
	Let $(M,\omega)$ be a $V$-symplectic manifold. A \emph{polysymplectic action} $\lambda:G\curvearrowright M$ is one which preserves $\omega$, that is, $\lambda(g)^*\omega=\omega$ for all $g\in G$. A \emph{polysymplectic vector field} $X\in\X(M)$ is one which is induced by a polysymplectic action.
\end{definition}

Thus, $X$ is polysymplectic precisely when $\L_X\omega = 0$. As in the classical context, we require a strengthening of this definition.

\begin{definition}
	Let $(M,\omega)$ be a $V$-symplectic manifold and suppose that $f\in C^\infty(M,V)$ and $X\in\X(M)$ satisfy
	\[
		-\iota_X\omega = \d f.
	\]
	Then $X$ is called the \emph{Hamiltonian vector field} of $f$, and $f$ is called the \emph{Hamiltonian function} of $X$. We also call $X$ the \emph{polysymplectic gradient} of $f$ and denote it by $s\text{-}\mathrm{grad}\,f$. More generally, $f$ (resp.\ $X$) is said to be \emph{Hamiltonian} if it possesses a Hamiltonian vector field (resp.\ Hamiltonian function).
\end{definition}

We will denote by $C_H^\infty(M,V)$ the space of Hamiltonian functions on $(M,\omega)$. Observe that $C_H^\infty(M,V)$ is a $C_H^\infty(M,V)$-symplectic vector space. We note that our Hamiltonian functions are termed \emph{currents} in the G\"unther's original paper \cite{Gunther87}.

In contrast with the classical case, it is not true in general that every function $f\in C^\infty(M,V)$ is Hamiltonian. This is an immediate consequence of the fact that the assignment
\begin{align*}
	\iota\,\omega_x:	T_xM		&\to		T_x^*M\otimes V,	\hspace{1cm}x\in M		\\
						X\;		&\mapsto		\,\iota_X\omega_x
\end{align*}
is never an isomorphism when $\dim V\geq 2$. A function $f\in C^\infty(M,V)$ is Hamiltonian if and only if $\d f_x$ lies in the image of $\iota\,\omega_x:T_xM\hookrightarrow\Hom(T_xM,V)$ at every point $x\in M$. The vector field $X\in\X(M)$ is Hamiltonian when $(\iota\,\omega)^{-1}X\in \Omega^1(M,V)$ is exact.

\begin{example}
\begin{enumerate}[A.]
	\item The function $(f_1,\ldots,f_N)\in C^\infty(M,\R^N)$ is Hamiltonian with respect to $\oplus_i\omega_i$, with Hamiltonian vector field $X\in\X(M)$, if and only if $X$ is the Hamiltonian vector field of $f_i$ with respect to $\omega_i$ for each $i\leq N$.
	\item Fix $\xi\in\g$ and let $\bar{\xi}\in\X(G)$ be the \emph{right} invariant extension of $\xi$ to $\X(G)$. Since the integral flow of $\bar{\xi}$ is realized by left multiplication by $\exp(t\xi)$, and since we have shown left multiplication to preserve $\theta$, it follows that $\L_{\bar{\xi}}\theta = 0$. Therefore,
	\[
		-\iota_{\bar{\xi}}\omega = \iota_{\bar{\xi}}\d\theta = -\d\theta(\bar{\xi}).
	\]
	Now, for every $g\in G$,
	\[
		\theta_g(\bar{\xi}) = (\lambda_{g^{-1}}\rho_g)_*\xi = \Ad_g^{-1} \xi,
	\]
	and thus the function
	\begin{align*}
		\Ad^{-1}\,\xi:	G	&\longrightarrow		\:\g				\\
						g	&\longmapsto			\Ad_g^{-1} \xi
	\end{align*}
	is Hamiltonian, with associated Hamiltonian vector field $-\bar{\xi}$.
	\item The steady rotation of $B$ about a fixed axis $\ell\subseteq S$ induces a Hamiltonian vector field on $Q_B$ with Hamiltonian function the angular velocity of $B$ in the frame $q\in Q_B$. We defer to Example \ref{eg:vv_hamiltonian_systems} for further details.
	\item The Hamiltonian functions on $\Hom(TQ,V)$ include the lifts $\pi^*f:\Hom(TQ,V)$ of smooth functions $f:Q\to V$, with associated Hamiltonian vector fields precisely the vertical vector fields on the bundle $\Hom(TQ,V)$ over $Q$.
\end{enumerate}
\end{example}

\begin{definition}
	The \emph{bracket}
	\[
		\{\,,\}:C_H^\infty(M,V)\times C_H^\infty(M,V)\to C_H^\infty(M,V)
	\]
	is defined on the space of Hamiltonian functions $C_H^\infty(M,V)$ by
	\[
		\{f,f'\} = -\omega(X_f,X_{f'}),
	\]
	for $f,f'\in C_H^\infty(M)$.
\end{definition}

\begin{remark}
The operation $\{\,,\}$ is a Lie bracket on $C_H^\infty(M,V)$, with respect to which the polysymplectic gradient map $f\mapsto X_f$ describes a Lie algebra antihomomorphism, and satisfies the property that
\[
	\{f,sf'\} = X_f(s)f' + s\,\{f,f'\}, \hspace{.8cm}f,f'\in C_H^\infty(M,V),\;s\in C^\infty(M),
\]
whenever $sf'\in C_H^\infty(M,V)$. However, it is noted in \cite{LucasVilarino15} that there is not in general a natural associative product on $C_H^\infty(M,V)$ with respect to which $\{\,,\}$ would be a Poisson bracket.
\end{remark}

\begin{definition}
	Let $\lambda$ be a polysymplectic action of a Lie group $G$ on the $V$-symplectic manifold $(M,\omega)$. A \emph{weak comoment map} is any linear map
	\[
		\tilde\mu: \g \to C_H^\infty(M,V)
	\]
	that lifts the fundamental vector fields of $\lambda$ to the space of Hamiltonian functions $C_H^\infty(M,V)$, as indicated in the following diagram.

	\vspace{.5cm}
	\begin{center}
	\begin{tikzcd}
																&C_H^\infty(M,V)\ar[d,"s\text{-}\mathrm{grad}"]		\\
		\g\ar[r,"\lambda_*" below]\ar[ur,"\tilde\mu",dashed]		&\X(M)
	\end{tikzcd}
	\end{center}
	\vspace{.5cm}
	If $\tilde\mu$ is additionally a morphism of Lie algebras, then it is called a \emph{comoment map}. The \emph{(weak) moment map} associated to a (weak) comoment map $\tilde\mu$ is the smooth function
	\[
		\mu : M \to \Hom(\g,V)
	\]
	given by
	\[
		\mu(x)(\xi) = \tilde\mu(\xi)(x),
	\]
	for $x\in M$ and $\xi\in\g$. When the action of $G$ admits a moment map $\mu$, the action is said to be \emph{Hamiltonian} and the quadruple $(M,\omega,G,\mu)$ is called a $V$-\emph{valued Hamiltonian system} or a $V$-\emph{Hamiltonian system}.
\end{definition}

\begin{proposition}\label{prop:exact_symplectic_form_moment_map}
	If the action of a Lie group $G$ on an exact $V$-symplectic manifold $(M,-\d\theta)$ preserves the polysymplectic potential $\theta\in\Omega^1(M,G)$, then the function $\mu_\theta:M\to\Hom(\g,V)$ given by
	\[
		\mu_\theta(x)(\xi) = \theta_x(\underline\xi_x),
	\]
	for $x\in M$ and $\xi\in\g$ is a moment map.
\end{proposition}

\begin{proof}
	Since $G$ preserves $\theta$, we have
	\[
		-\iota_{\underline\xi}(-\d\theta) = \L_{\underline\xi}\theta - \d\iota_{\underline\xi}\theta = \d\big[\!-\theta(\underline\xi)\big]
	\]
	for all $\xi\in\g$, and it follows that $-\theta(\underline\xi)\in C^\infty(M,V)$ is a Hamiltonian function for the vector field $\underline\xi\in\X(G)$. Since
	\begin{align*}
		&\theta\big(\underline{[\xi,\eta]}\big) = -\L_{\underline\xi}\iota_{\underline\eta}\theta = -\iota_{\underline\xi}\d\iota_{\underline\eta}\theta		\\
			&\hspace{1.1cm}=\iota_{\underline\xi}\iota_{\underline\eta}\d\theta = -\d\theta\big(\underline\xi,\underline\eta\big) = \{\theta(\underline\xi),\theta(\underline\eta)\},
	\end{align*}
	we deduce that the assignment $\underline\xi\mapsto\iota_{\underline\xi}\theta$ is a comoment map.
\end{proof}

We catalog the foregoing constructions beside their classical counterparts in the table below.
\begin{center}
\begin{tabular}{rccc}
							&classical						&$V$-valued							&							\\ \cline{1-3}
							&								&									&							\\ [-.3cm]
		symplectic form		&$\omega\in\Omega^2(M,\R)$		&$\omega\in\Omega^2(M,V)$			&							\\
		Hamiltonian function	&$f\in C^\infty(M)$				&$f\in C^\infty(M,V)$				&$\omega(\,\cdot,X_f)=df$	\\
		comoment map		&$\tilde\mu:\g\to C^\infty(M)$	&$\tilde\mu:\g\to C^\infty(M,V)$	&							\\
		moment map			&$\mu:M\to \g^*$					&$\mu:M\to \Hom(\g,V)$				&
\end{tabular}
\end{center}

\begin{definition}
	The \emph{coadjoint action} $\Ad^*:G\curvearrowright\Hom(\g,V)$ is given by
	\[
		(\Ad_g^* \alpha)(\xi) = \alpha(\Ad_g^{-1}\xi),
	\]
	for $g\in G$, $\alpha\in\Hom(\g,V)$, and $\xi\in\g$.
\end{definition}

\begin{lemma}\label{lem:vv_hamiltonian_system_facts}
	Let $(M,\omega,G,\mu)$ be a $V$-Hamiltonian system.
	\begin{enumerate}[i.]
		\item If $\alpha\in\Hom(\g,V)$, then the assignment
			\begin{align*}
				\mu+\alpha:	M	&\to	\Hom(\g,V)		\\
							x	&\to	\mu(x)+\alpha
			\end{align*}
		is a weak moment map. In particular, the set of weak moment maps compatible with $(M,\omega,G)$ is a $\Hom(\g,V)$-affine space.
		
		If additionally $\alpha$ vanishes on commutators $[\xi,\eta]\in\g$ ($\xi,\eta\in\g$) then $\alpha$ is a moment map. Consequently, the set of moment maps is a $[\g,\g]^0$-affine space, where $[\g,\g]^0$ denotes the annihilator of $[\g,\g]\subseteq\g$ in $\Hom(\g,V)$.
		\item If $G$ is connected, then the map $\mu:M\to\Hom(\g,V)$ is a moment map for the action of $G$ on $(M,\omega)$ precisely when
			\[
				\langle \mu_*X,\xi\rangle = \omega(X,\underline{\xi}_x)
			\]
			for all $x\in M$, $X\in T_xM$, and $\xi\in\g$. Here $\langle\,,\rangle:\Hom(\g,V)\times\g\to V$ denotes the natural pairing and $\underline{\xi}_x$ is the value of the action-induced vector field for $\xi$ at $x$.
		\item If $G$ is connected, then $\mu$ intertwines the action of $G$ on $M$ with the coadjoint action of $G$ on $\Hom(\g,V)$.
		\item If $H\subseteq G$ is a Lie subgroup, and if the map $\mu|_\h$ is given by
			\begin{align*}
				\mu|_\h:		M	&\to		\Hom(\h,V)		\\
							x	&\mapsto		\,\mu(x)|_\h,
			\end{align*}
		 then $(M,\omega,H,\mu|_\h)$ is a $V$-Hamiltonian system.
	\end{enumerate}
\end{lemma}

The proof of these assertions are so similar to their classical analogues that we omit the proofs and refer instead to the corresponding symplectic literature \cite{Cannas-da-Silva01,Meinrenken00}.

\begin{example}\label{eg:vv_hamiltonian_systems}
\begin{enumerate}[A.]
	\item The action of $G$ on $(M,\oplus_i\omega_i)$ is Hamiltonian if and only if it is Hamiltonian with respect to each $\omega_i$ for $i\leq N$. In this case, a moment map is given by $\oplus_i\mu_i:M\to(\g^*)^N\cong\Hom(\g,\R^N)$.
	\item Since the fundamental vector fields of the left regular representation of $G$ are the right invariant vector fields on $G$, Proposition \ref{prop:exact_symplectic_form_moment_map} implies that the map
	\[
		\mu : G \to \End\,\g
	\]
	given by
	\[
		\mu(g)(\xi) = \theta_g(\underline\xi_g) = (\lambda_{g^{-1}})_*(\rho_g)_*\xi = \Ad_g^{-1}\,\xi
	\]
	is a moment map for the left regular representation of $G$. Here we denote by $\lambda$ and $\rho$ the left and right regular representations, respectively.
	\item The induced action on $\Hom(TQ,V)$ of a subgroup $G\subseteq\Aut\,Q$ is Hamiltonian with canonical moment map $\mu:\Hom(TQ,V)\to V$ given by $\mu(\phi)(\xi) = \theta_\phi(\underline\xi_\phi) = \phi(\underline\xi_q)$ where $q=\pi\phi\in Q$.
\end{enumerate}
\end{example}

We recall the Arnold conjecture for compact classical symplectic manifolds $(M,\omega)$.

\begin{conjecture}[Arnold \cite{McDuffSalamon98}]
		A symplectomorphism that is generated by a time-dependent Hamiltonian vector field should have at least as many fixed points as a Morse function on the manifold must have critical points.
\end{conjecture}

From Example \ref{eg:vv_hamiltonian_systems} we deduce the following.

\begin{theorem}
	The Arnold conjecture fails in the $V$-symplectic setting.
\end{theorem}

\begin{proof}
	Consider the $\g$-symplectic manifold $(G,-\d\theta)$, where $G$ is a compact semisimple Lie group and $\theta\in\Omega^1(G,\g)$ is the Maurer-Cartan form. Let $T\subseteq G$ be a $1$-torus, let $\xi\in\g$ be a generator of $T$ with $\exp(\xi)=1$, and let $\bar\xi\in\X(G)$ be the right invariant extension of $\xi$. The $1$-periodic family of polysymplectomorphisms
	\begin{align*}
		\phi_t:	G	&\to			G						\\
				g	&\mapsto		e^{t\xi}g
	\end{align*}
	is generated by the Hamiltonian vector fields $t\bar\xi$. When $e^{t\xi}\neq 1$ the transformation $\phi_t$ is fixed-point free. However, any nondegenerate function on the compact space $G$ has at least two critical points.
\end{proof}


\subsection{$V$-Hamiltonian Reduction}

We begin with the fundamental theorem of $V$-Hamiltonian systems, first stated in \cite{Gunther87} and proved in \cite{MarreroRoman-RoySalgado15}.

\begin{theorem}[Vector-Valued Hamiltonian Reduction]\label{thm:vv_symplectic_reduction}
	Let $(M,\omega,G,\mu)$ be a $V$-Hamiltonian system and fix $\alpha\in\Hom(\g,V)$. If the stabilizer subgroup $G_\alpha$ of $\alpha$ under the coadjoint action is connected, and if $M_\alpha=\mu^{-1}(\alpha)/G_\alpha$ is smooth, then there is a unique $V$-valued $2$-form $\omega_\alpha\in\Omega^2(M_\alpha,V)$ such that
	\[
		\pi^*\omega_\alpha = i^*\omega,
	\]
	where $i:\mu^{-1}(\alpha)\hookrightarrow M$ is the inclusion and $\pi:\mu^{-1}(\alpha)\to M_\alpha$ is the projection. The form $\omega_\alpha$ is closed and is nondegenerate at $\pi x$ if and only if $\underline{\g_\alpha}_x = \underline\g_x^{\omega\omega}\cap\underline\g_x$.
\end{theorem}

\begin{proof}
	First note that the equivariance of $\mu$ ensures that the action of $G_\alpha$ preserves $\mu^{-1}(\alpha)$, and thus that the quotient $\mu^{-1}(\alpha)/G_\alpha$ exists as a topological space.

	Fix $x\in \mu^{-1}(\alpha)$. Lemma \ref{lem:vv_hamiltonian_system_facts} implies that
	\[
		X\in\underline\g_x^\omega		\; \iff\;	\omega(X,\underline{\g}_x) = \langle\mu_*X,\g\rangle=0		\;\iff\;	\mu_*X=0
	\]
	for all $X\in T_xM$, so that
	\[
		\underline\g_x^\omega = T_x\mu^{-1}(\alpha).
	\]
	Therefore, Theorem \ref{thm:linear_vv_symplectic_reduction} implies that $\omega_x$ descends to a bilinear form on
	\[
		T_x\mu^{-1}(\alpha)/\underline{\g_\alpha}_x = \underline{\g}_x^\omega/(\underline{\g}_x\cap\underline{\g}_x^\omega)
	\]
	with kernel $(\underline\g_x^{\omega\omega}\cap\underline\g_x)/\underline{\g_\alpha}_x$.

	Since $T_{\pi x}M_\alpha \cong T_x\mu^{-1}(\alpha)/\underline{\g_\alpha}_x$, we obtain a $2$-form $\omega_\alpha\in \Omega^2(M_\alpha,V)$ with $\pi^*\omega_\alpha = i^*\omega$.
	
	As $\pi$ is surjective, the induced map $\pi^*$ is injective and $\omega_\alpha$ is unique. Closedness follows by the injectivity of $\pi^*$ and the equality $\pi^*\d\omega_\alpha = \d\pi^*\omega_\alpha = 0$.
\end{proof}

We call $(M_\alpha,\omega_\alpha)$ the \emph{reduction} of $(M,\omega,G,\mu)$ at level $\alpha$, with \emph{reduced space} $M_\alpha$ and \emph{reduced $2$-form} $\omega_\alpha$.

\begin{remark}
	When $\alpha=0$, the distribution $\underline{\g_0} = \underline\g$ is isotropic along $\mu^{-1}(0)$ and the condition for the nondegeneracy of $\omega_0$ at $\pi x\in M_0$ becomes $\underline\g_x = \underline\g_x^{\omega\omega}$.
\end{remark}

\begin{remark}
	There is another approach to reduction, which is equivalent to ours in the classical case, but which diverges for more general coefficients $V$. Given a $V$-Hamiltonian system $(M,\omega,G,\mu)$ and a level $\alpha\in\Hom(\g,V)$, the restriction $i^*\omega$ of $\omega$ to $\mu^{-1}(\alpha)$ is a closed form. Now, the kernel distribution of any closed form $\sigma\in\Omega^*(M,V)$ is integrable, as $X,Y\in \ker\sigma$ implies that
	\[
		\iota_{[X,Y]}\sigma = (\L_X\iota_Y -\iota_Y\L_X)\sigma = (\L_X\iota_Y -\iota_Y\iota_X\d-\iota_Y\d\iota_X)\sigma = 0.
	\]
	When it is smooth, the leaf space $\tilde{M}_\alpha=\mu^{-1}(\alpha)/\ker i^*\omega$ naturally inherits a $V$-symplectic structure $(\tilde{M}_\alpha,\tilde\omega_\alpha)$, where $\tilde\omega_\alpha$ is the unique $2$-form on $\tilde{M}_\alpha$ satisfying $\tilde\pi^*\tilde\omega_\alpha = i^*\omega$. From the proof of Theorem \ref{thm:vv_symplectic_reduction}, it is apparent that $\tilde{M}_\alpha$ is a quotient of the reduced space $M_\alpha$. When the reduced $2$-form $\omega_\alpha$ is polysymplectic, as is always the case in the classical symplectic setting, the spaces $(\tilde{M}_\alpha,\tilde\omega_\alpha)$ and $(M_\alpha,\omega_\alpha)$ coincide.
	
	We refer to \cite{MarreroRoman-RoySalgado15} for further details.
\end{remark}

\begin{example}\label{eg:vv_symplectic_reduction}
\begin{enumerate}[A.]
	\item The reduced space of the $\R^N$-Hamiltonian system $(M,\oplus_i\omega_i,G,\oplus_i\mu_i)$ is the intersection of the reduction of each $(M,\omega_i,G,\mu_i)$. That is,
	\[
		M_\alpha = \Big(\bigcap_i\, \mu_i^{-1}(\alpha)\Big)/G_\alpha = \bigcap_i \big(\mu_i^{-1}(\alpha)/G_\alpha\big).
	\]
	\item Let $G$ be a Lie group with discrete center, let $\theta\in\Omega^1(G,\g)$ be the Maurer-Cartan form, and $H\subseteq G$ be a connected Lie subgroup of $G$. Since the left regular action of $H$ on $G$ is Hamiltonian with moment map $\mu=\Ad^{-1}|_\h:G\to\Hom(\h,\g)$. Since the adjoint representation acts by automorphisms, the preimage of $0\in\Hom(\h,\g)$ under $\mu$ is empty, and thus the reduced space $G_0$ is empty as well.
	
	Let us compute the reduction of $(G,-\d\theta)$ at the inclusion $i:\h\hookrightarrow\g$. We have
	\[
		g\in\mu^{-1}(i)	\,\; \iff\;\;	\forall{\xi\in\h}:\Ad_{g^{-1}}\,\xi = \xi	\;\;\iff\;\,	g\in C_G(H),
	\]
	where $C_G(H)$ is the centralizer of $H$ in $G$, and where the second equivalence follows as $H$ is connected. Denote by $H_i\subseteq H$ the stabilizer subgroup of $i$ under the coadjoint action of $H$ on $\Hom(\h,\g)$. It follows that
	\[
		h\in H_i \,\; \iff\;\;	\forall{\xi\in\h}:\Ad_h\,\xi = \xi	\;\;\iff\;\,	h\in Z(H),
	\]
	where $Z(H)\subseteq H$ is center of $H$. We conclude that the reduced space is
	\[
		G_i = \mu^{-1}(i)/H_i = C_G(H)/Z(H).
	\]
	Since $Z(H)$ is a central subgroup of $C_G(H)$, it follows that
	\[
		 G_i = C_G(H)/Z(H) = C_G(H) / \big(H \cap C_G(H)\big)
	\]
	as a normal quotient of groups. Let $\bar\theta_i\in\Omega^1(G_i,\mathfrak{c}_\g(\h)/\mathfrak{z}(\h))$ denote the Maurer-Cartan form on $G_i$. When $G_i$ has discrete center, $-\d\bar\theta_i$ is a $\mathfrak{c}_\g(\h)/\mathfrak{z}(\h)$-symplectic form, and is the image of a reduced potential $\theta_i\in\Omega^1(G_i,\g)$, i.e.\ $\omega_i=-\d\theta_i$, under a reduction of coefficients $f:\g\to\mathfrak{c}_\g(\h)/\mathfrak{z}(\h)$. In particular, the reduction $(G_i,\omega_i)$ is polysymplectic when $C_G(H)/Z(H)$ has discrete center.
	\item Using the fact that $(Q_B,\omega_B)$ and $(\SO(3),-\d\theta)$ are weakly symplectomorphic, it follows that the reduction of $(Q_B,\omega_B,T,\mu)$ at any level $\alpha\in\Hom(\t,\R^3)$ is either empty or a point.
	\item The reduction of $\Hom(TQ,V)$ by a subgroup $G\subseteq \Diff\,Q$ at level $0\in\Hom(\g,V)$ is naturally isomorphic to $\Hom\big(T(Q/G),V\big)$. In particular, the reduction is polysymplectic.
\end{enumerate}
\end{example}

\begin{remark}
	In contrast with the classical situation \cite{Atiyah82,GuilleminSternberg82,Kirwan84}, Example \ref{eg:vv_symplectic_reduction}.B.\ shows that the image of a moment map $\mu:M\to\Hom(\t,V)$ for the Hamiltonian action of a torus $T\subseteq G$ is not necessarily convex.
\end{remark}

\begin{proposition}\label{prop:Lagrangian_reduction}
	If $(M,\omega,G,\mu)$ is a $V$-Hamiltonian system, and if the reduced $2$-form $\omega_0$ vanishes on $M_0$, then the regular part of $\mu^{-1}(0)$ is a Lagrangian submanifold of $M$.
\end{proposition}

\begin{proof}
	Let $x\in\mu^{-1}(0)$. Since $\omega_x$ descends to zero on $T_x\mu^{-1}(0)/ \underline{\g}_x$ it follows that $\omega_x$ vanishes on $T_x\mu^{-1}(0)$, from which $T_x\mu^{-1}(0) \subseteq T_x\mu^{-1}(0)^\omega$. Taking the polysymplectic orthogonal of both sides of the inclusion
	\[
		\underline{\g}_x\subseteq \underline{\g}_x^\omega = T_x\mu^{-1}(0)
	\]
	yields
	\[
		T_x\mu^{-1}(0) = \underline{\g}_x^\omega\supseteq \underline{\g}_x^{\omega\omega} = T_x\mu^{-1}(0)^{\omega}.
	\]
	Thus, $T_x\mu^{-1}(0) = T_x\mu^{-1}(0)^\omega$.
\end{proof}

We complete this section with a result on the topology of the reduced space.

\begin{theorem}
	Let $(M,\omega,G,\mu)$ be a $V$-Hamiltonian system with compact Lie group $G$, and suppose that $\alpha\in\Hom(\g,V)$ is a regular value of the moment map $\mu:M\to\Hom(\g,V)$. Then the reduced space $M_\alpha$ has at most orbifold singularities.
\end{theorem}

\begin{proof}
	Since $\alpha$ is a regular value, $\mu^{-1}(\alpha)\subseteq M$ is a smooth manifold. Fix $x\in\mu^{-1}(\alpha)$ and let $G_x\subseteq G$ be the stabilizer subgroup of $x$, with Lie algebra $\g_x$. Since $\alpha$ is a regular value of $\mu$ it follows that $\mu_*T_xM=\Hom(\g,V)$ and thus
	\[
		\big\langle\Hom(\g,V),\g_x\big\rangle = \omega\big(\mu_*T_xM,\underline{\g_x}_x\big) = 0.
	\]
	Consequently, $\g_x=0$ and $G_x$ is discrete. We conclude that the stabilizer subgroup $(G_\alpha)_x\subseteq G_x$ is discrete as well and the quotient $\mu^{-1}(\alpha)/G_\alpha$ has at most orbifold singularities.
\end{proof}

We note that in the finite-dimensional classical symplectic situation the converse is also true: the level $\alpha\in\Hom(\g,V)$ is a regular value of $\mu$ if and only if $M_\alpha$ possesses an orbifold structure.


\section{Gauge Theory in Higher Dimensions}\label{sec:gauge_theory_in_higher_dimensions}

Atiyah and Bott observed \cite[Section 9]{AtiyahBott83} that, for a compact connected group $G$ with $\Ad$-invariant metric $\langle\,,\rangle_\g$ on the Lie algebra $\g$, the moduli space of flat connections on a $G$-principal bundle $P$ over a surface $\Sigma$ is obtained as the symplectic reduction of the space of connections $\A$ on $P$ by the action of the gauge group $\G = \Gamma(\Sigma,\Ad P)$, where $\Ad P = P\times_c G$ and $c:G\to \Aut\,G$ is the action of conjugation. More precisely, there is a natural symplectic structure on $\A$ given by
\[
	\omega_A(\alpha,\beta) = \int_\Sigma \langle\alpha\wedge\beta\rangle,
\]
where $A\in\A$, $\alpha,\beta\in\Omega^1(\Sigma,\ad P)\cong T_A\A$, $\ad P = P\times_{\mathrm{Ad}}\g$ is the adjoint bundle of $P$, and we define the operation
\[
	\langle\;\;\wedge\;\;\rangle:\Omega^1(\Sigma,\ad P)\times\Omega^1(\Sigma,\ad P) \xrightarrow{\wedge_{\Omega^1(\Sigma)}} \Omega^2(M,\ad P\otimes\ad P)
		\xrightarrow{\langle\,,\,\rangle_{\ad P} } \Omega^2(\Sigma).
\]
Here the metric $\langle\,,\rangle_{\ad P}$ on $\ad P$ is induced by the $\Ad$-invariant metric on $\g$. Writing $\gcal$ for the Lie algebra of $\G$, a moment map $\mu:\A\to\gcal^*$ for the induced action of $\G$ on $\A$ is given by
\[
	\mu(A)(f) = \int_\Sigma \langle F_A\wedge f\rangle,
\]
where $F_A\in\Omega^2(\Sigma,\ad P)$ is the curvature of $A\in\A$, and $f\in\Omega^0(\Sigma,\ad P)\cong\gcal$. The reduction of the Hamiltonian system $\big(\A,\omega,\G,\mu\big)$ at the level $0\in\gcal^*$ is the moduli space $\M(P)$ of flat connections on $P$.

The main result of this paper is that there is a similar polysymplectic characterization of the moduli space of flat connections over a higher dimensional manifold $M$. For clarity of exposition, we first examine the relatively simple case given by the first cohomology $\Omega^1(M)$ of $M$, which we equip with a formal $\Omega^2(M)/B^2(M)$-symplectic structure. We then proceed to establish the main result. We complete this section with an application of the polysymplectic reduction procedure to a family of degenerate $2$-forms arising from Chern-Weil theory.


\subsection{Cohomology as $\Omega^2(M)/B^2(M)$-Symplectic Reduction}\label{subsec:cohomology_as_symplectic_reduction}

The purpose of this subsection is to provide an accessible simplification of the gauge theoretic material in Subsection \ref{subsec:reduction_of_the_space_of_connections}. As our objective is to elucidate later developments, our presentation will be \emph{formal} in the sense that we will not address the subtleties that arise from the consideration of infinite dimensional manifolds. However, we note that this material is readily adapted to the setting of Banach manifolds, in precisely a manner analogous to the presentation in Subsection \ref{subsec:reduction_of_the_space_of_connections}. Indeed, with the exception that the Lie group $G=(\R,+)$ is not semisimple, we may regard this material as a special case of the theory of connections and principal bundles as it is treated in Subsection \ref{subsec:reduction_of_the_space_of_connections}.

It is noted in \cite{AtiyahBott83} that the vector space $\Omega^1(\Sigma)$ carries a natural symplectic structure: namely,
\[
	\omega(\alpha,\beta) = \int_\Sigma \alpha\wedge\beta,		\hspace{1cm}\alpha,\beta\in\Omega^1(\Sigma).
\]
Our present aim is to adapt this symplectic structure to the case in which $\dim M\geq 3$.

The most natural polysymplectic structure on $\Omega^1(M)$ is the wedge product $\wedge$. The following proposition establishes that $\wedge$ is indeed a $\Omega^2(M)$-symplectic structure on $\Omega^1(M)$.

\begin{proposition}
	Let $M$ be a manifold of dimension at least $2$. The wedge product
	\[
		\wedge:\Omega^1(M)\times\Omega^1(M) \to \Omega^2(M)
	\]
	is a formal $\Omega^2(M)$-symplectic structure on the vector space $\Omega^1(M)$.
\end{proposition}

\begin{proof}
	As $\wedge$ is clearly a skew-symmetric $\Omega^2(M)$-valued form on $\Omega^1(M)$, we have only to show that it is nondegenerate. Thus let $\alpha\in \Omega^1(M)$ and suppose that $\alpha\wedge\beta=0$ for all $\beta\in\Omega^1(M)$. Let $(x^i)_{i\leq n}$ ($n=\dim M$) be a system of coordinates on a neighborhood $U\subseteq M$, and let $\alpha_i\in C^\infty(U)$ be given by
	\[
		\alpha = \sum_i \alpha_i \,\d x^i.
	\]
	For each $k\leq n$,
	\[
		0 = \alpha \wedge \d x^k = \sum_i \alpha_i \,\d x^i \wedge \d x^k.
	\]
	Since $n\geq 2$, for each $i\leq n$ there is a $k\leq n$ with $k\neq i$, and thus $\d x^i\wedge\d x^k \neq 0$ so that $\alpha_i=0$. Since our choice of $U$ was arbitrary, we conclude that $\alpha = 0$.
\end{proof}

It turns out that this polysymplectic structure is too fine for our purposes, in the sense that the action of $C^\infty(M)$ on $\Omega^1(M)$ given by
\[
	f\cdot \alpha = \d f + \alpha
\]
is not generally Hamiltonian with respect to polysymplectic structure $\omega$ obtained by lifting $\wedge$ to the fibers of $T\Omega^1(M)$. The issue is resolved by reducing the space of coefficients from $\Omega^2(M)$ to $\Omega^2(M)/B^2(M)$, where $B^2(M)=\d\Omega^1(M)$ is the space of $2$-coboundaries on $M$.

We first establish a technical lemma.

\begin{lemma}\label{lem:nondegeneracy_of_wedge_in_higher_than_three_dimensions}
\begin{enumerate}[i.]
	\item \label{lem:nondegeneracy_of_wedge_in_higher_than_three_dimensions_i} Let $U$ be a vector space with $\dim U\geq 3$ and let $w\in\Lambda^2U$. If $u\wedge w = 0$ for all $u\in U$ then $w = 0$.
	\item \label{lem:nondegeneracy_of_wedge_in_higher_than_three_dimensions_ii} Let $M$ be a manifold with $\dim M\geq 3$. If $\theta\in\Omega^2(M)$ satisfies $\d(f\theta)=0$ for all $f\in C^\infty(M)$, then $\theta = 0$.
\end{enumerate}
\end{lemma}

\begin{proof}
\begin{enumerate}[i.]
	\item Fix a basis $\{e_i\}_{i\leq n}$ of $U$ and choose coefficients $w^{ij}\in\R$ so that
		\[
			w = \sum_{i,j\leq n} w^{ij}\, e_i\wedge e_j.
		\]
		For each $k\leq n$, we have
		\[
			0 = e_k\wedge\omega = \sum_{i,j\leq n} w^{ij} e_k\wedge e_i\wedge e_j.
		\]
		Since $n\geq 3$, for every pair of distinct $i,j\leq n$ we can find a $k\leq n$ with $k\neq i,j$. Consequently, $e_k\wedge e_i\wedge e_j\neq 0$ and thus $w^{ij}=0$.
	\item As $\d(1\cdot\theta) = 0$,
		\[
			\d f\wedge \theta = \d(f\theta) = 0
		\]
		for all $f\in C^\infty(M)$. Fix $p\in M$ and observe that $\alpha\wedge \theta_p = 0\in \Lambda^3(T_p^*M)$ for all $\alpha=\d f_p\in T_p^*M$. Now part \ref{lem:nondegeneracy_of_wedge_in_higher_than_three_dimensions_i} yields $\theta_p = 0$.
\end{enumerate}
\end{proof}

\begin{proposition}\label{prop:infinite_dimensional_vector_valued_symplectic_vector_space_toy_model}
	Let $M$ be a compact manifold of dimension at least $3$. The assignment
	\[
		\omega:\Omega^1(M)\times\Omega^1(M)\to\Omega^2(M)/B^2(M)
	\]
	defined by
	\[
		\omega(\alpha,\beta) = (\alpha\wedge\beta)_{\Omega^2/B^2},		\hspace{.8cm}\alpha,\beta\in\Omega^1(M),
	\]
	is a formal $\Omega^2(M)/B^2(M)$-symplectic structure on the vector space $\Omega^1(M)$.
\end{proposition}

\begin{proof}
	Let $\alpha\in\Omega^1(M)$, and assume that $\alpha\wedge\gamma \in B^2(M)$ for all $\gamma\in \Omega^1(M)$. Let $\beta\in\Omega^1(M)$ and observe that
	\[
		\d(\alpha\wedge f\beta) = \d(f\;\alpha\wedge\beta) \in B^2(M)
	\]
	for all $f\in C^\infty(M)$. Thus Lemma \ref{lem:nondegeneracy_of_wedge_in_higher_than_three_dimensions} implies that $\alpha\wedge\beta = 0$. Since our choice of $\beta$ was arbitrary, the nondegeneracy of the wedge product on $\Omega^1(M)$ yields $\alpha = 0$.
\end{proof}


Let $\omega$ be the $\Omega^2(M)/B^2(M)$-valued $2$-form on $\Omega^1(M)$ be given by
\[
	\omega_A(\alpha,\beta) = (\alpha\wedge\beta)_{\Omega^2/B^2},\hspace{1cm}A\in\Omega^1(M),\;\alpha,\beta\in\Omega^1(M)\cong T_A\Omega^1(M),
\]
and let $C^\infty(M)$ act on $\Omega^1(M)$ as
\[
	f\cdot \alpha = \d f + \alpha.
\]
The aim of this section is to show that the polysymplectic reduction of $\big(\Omega^1(M),\omega\big)$ is the first cohomology $H^1(M)$ with the wedge product $\wedge_{H^1}$.

\begin{proposition}
	The space $\big(\Omega^1(M),\omega\big)$ is a formal $\Omega^2(M)/B^2(M)$-symplectic manifold.
\end{proposition}

\begin{proof}
	Closedness follows as $\omega$ is constant on $\Omega^2(M)$. For every $A\in\Omega^1(M)$, Proposition \ref{prop:infinite_dimensional_vector_valued_symplectic_vector_space_toy_model} ensures that the restriction of $\omega$ to the fiber $\Omega^1(M)\cong T_A\Omega^1(M)$ of the tangent bundle $T\Omega^1(M)$ is nondegenerate.
\end{proof}

\begin{theorem}
	The action of $C^\infty(M)$ on $\Omega^1(M)$, given by $f\cdot \alpha = \d f+\alpha$, is Hamiltonian with respect to $\omega$. A moment map
	\[
		\mu:\Omega^1(M)\to \Hom\big(C^\infty(M),\:\Omega^2(M)/B^2(M)\big)
	\]
	is given by
	\[
		\mu(A)(f) = (\d A\wedge f)_{\Omega^2/B^2}.
	\]
	The reduced space is $\big(H^1(M),\omega_0\big)$ where, for each cohomology class $\bar{A}=A+B^1(M)\in H^1(M)$, the $2$-form
	\[
		\omega_0 : T_{\bar A} H^2(M)\times T_{\bar A}H^2(M) \longrightarrow H^2(M)\subseteq \Omega^2(M)/B^2(M)
	\]
	is given by
	\[
		\omega_0\big(\bar\alpha_{\bar A},\bar\beta_{\bar A}\big) = \bar\alpha\wedge\bar\beta, \hspace{1cm}\bar\alpha,\bar\beta\in T_{\bar A}H^2(M) \cong H^2(M).
	\]
\end{theorem}

\begin{proof}
	For $\alpha\in\Omega^1(M)$, $f\in C^\infty(M)$, the equality
	\[
		\d(\alpha\wedge f) = \d\alpha\wedge f - \alpha\wedge\d f
	\]
	implies that
	\[
		(\d\alpha\wedge f)_{\Omega^2/B^2} = (\alpha\wedge\d f)_{\Omega^2/B^2}.
	\]
	Since $\d:\Omega^1(M)\to\Omega^2(M)$ is linear, the induced map $\d_*:T\Omega^1(M)\to T\Omega^2(M)$ is given by
	\[
		\d_*\alpha_A = (\d\alpha)_{\d A} \in T_{\d A}\Omega^2(M)
	\]
	for every $A\in\Omega^1(M)$ and $\alpha\in\Omega^1(M)\cong T_A\Omega^1(M)$. Thus,
	\begin{align*}
		\big\langle \mu_*\alpha_A, f\big\rangle
			&=	(\d\alpha\wedge f)_{\Omega^2/B^2}			\\
			&=	(\alpha\wedge\d f)_{\Omega^2/B^2}			\\
			&=	\omega(\alpha_A,\underline{f}_A),
	\end{align*}
	and it follows that $\mu$ is a moment map for the action of $C^\infty(M)$ on $\Omega^1(M)$. Since
	\[
		\mu(A)=0 \iff\; \d A\wedge C^\infty(M) \in B^2(M)  \;\iff \d A = 0,
	\]
	we conclude that the reduced space is $\mu^{-1}(0)/C^\infty(M) = Z^2(M)/B^2(M) = H^2(M)$.
\end{proof}


\subsection{The Reduction of the Space of Connections}\label{subsec:reduction_of_the_space_of_connections}

Let $M$ be a compact connected manifold of dimension at least $3$, let $G$ be a compact matrix Lie group with $\Ad$-invariant metric $\langle\,,\rangle_\g$ on the Lie algebra $\g$, let $P$ be a $G$-principal bundle on $M$ with connected gauge group $\G=\Gamma(M,\Ad\,P)$, let $\gcal$ be the Lie algebra of $\G$, and let $\A$ be the $\Omega^1(M,\ad P)$-affine space of connections on $P$. Here we recall that $\Ad\,P = P\times_cG$, where $c:G\to\Aut\,G$ is the action of conjugation, and $\ad P= P\times_{\Ad}\g$. We refer to \cite{KobayashiNomizu96} for further details on principle bundles and connections.

We will follow the original paper of Atiyah and Bott \cite{AtiyahBott83} and employ the theory of Banach manifolds, in particular, Sobolev spaces, to properly address the infinite-dimensional spaces that we encounter. Thus, fix an integer $k>\frac{1}{2}\dim M+1$. For a fiber bundle $S\to M$, we denote by $\Gamma_k(M,S)$ the $W^{k,2}$ Sobolev completion of the space of smooth sections $\Gamma(M,S)$, and we write $\tilde\Gamma(M,S)$ for the Banach space of $C^1$ sections of $S$ equipped with the $C^1$ norm. The action of $\Omega^1(M,\ad P)$ on $\A$ yields a Hilbert manifold $\A_k = \A+\Omega_k^1(M,\ad P)$ modeled on $\Omega_k^1(M,\ad P)$ \cite{Uhlenbeck82}. Note that $\Gamma_k$ (resp.\ $\A_k$) is continuously embedded in $\tilde\Gamma$ (resp.\ $C^1$ connections on $P$). We refer to \cite{Palais68} for background on Sobolev spaces, and \cite{Lang62} for Hilbert and Banach manifolds more generally. For a very interesting treatment of classical symplectic manifolds in the Banach space setting, we recommend \cite{Marsden67}.

Our particular choice of spaces is motivated by the following result.

\begin{theorem}[$\!$\cite{MitterViallet81}]
	The group $\G_{k+1}$ is a smooth Hilbert Lie group with Lie algebra $\gcal_{k+1}$ canonically isomorphic to $\Omega_{k+1}^0(M,\ad P)$, the action of $\G_{k+1}$ on $\A_k$ is smooth, and the vector field induced by $f\in\Omega_{k+1}^0(M,\ad P)\cong\gcal_{k+1}$ corresponds to $\d_Af\in\Omega_k^1(M,\ad P)\cong T_A\A_k$ at $A\in\A_k$.
\end{theorem}

The manifold $\A_k$ may be covered by a single chart, so that the tangent bundle $T\A_k\to\A_k$ is globally isomorphic to $\pi_1:\A_k\times\Omega_k^1(M,\ad P)\to \A_k$ in the sense of Banach manifolds \cite[Chapter III]{Lang85}. Indeed, since $\A_k$ possesses a natural affine action of $\Omega_k^1(M,\ad P)$, the isomorphism $T\A\cong\A\times\Omega_k^1(M,\ad P)$ is canonical, and we denote by $\alpha_A\in T\A_k$ the tangent vector corresponding to $(A,\alpha)\in\A_k\times\Omega_k^1(M,\ad P)$. The invariant metric $\langle\,,\rangle_\g$ induces a metric $\langle\,,\rangle_{\ad P}$ on the fibers of $\ad P$, with which we define the bilinear map
\[
	\langle\;\wedge\;\rangle:\Omega^j(M,\ad P)\times\Omega^{2-j}(M,\ad P) \xrightarrow{\wedge_{\Omega^1(M)}} \Omega^2(M,\ad P\otimes\ad P)
		\xrightarrow{\langle\,,\,\rangle_{\ad P} } \Omega^2(M),
\]
for $j=1,2$.

\begin{lemma}\label{lem:smooth_induced_product}
	Let $j=1$ or $2$.
	\begin{enumerate}[i.]
		\item The function $\langle\;\wedge\;\rangle$ is continuous from $\Omega_k^j(M,\ad P)\times\Omega_k^{2-j}(M,\ad P)$ to $\tilde\Omega^2(M)$.
		\item The spaces $\tilde{Z}^2(M)$ and $\tilde{B}^2(M)$ are closed in $\tilde\Omega^2(M)$. In particular, $\tilde\Omega^2(M)/\tilde{B}^2(M)$ and $H^2(M)=\tilde{Z}^2(M)/\tilde{B}^2(M)$ are naturally Banach spaces.
		\item The composition
		\[
			\langle\;\wedge\;\rangle_{\tilde\Omega^2/\tilde{B}^2} : \Omega_k^j(M,\ad P)\times\Omega_k^{2-j}(M,\ad P)
				\xrightarrow{\langle\;\wedge\;\rangle} \tilde\Omega^2(M)
				\xrightarrow{[\;\;]_{\tilde\Omega^2/\tilde{B}^2}} \tilde\Omega^2(M)/\tilde{B}^2(M)
		\]
		is a smooth map of Banach spaces.
	\end{enumerate}
\end{lemma}

\begin{proof}
	\begin{enumerate}[i.]
		\item Since $\langle\;\wedge\;\rangle$ is induced by a smooth map from $\Lambda^jT^*M\otimes\ad P\times \Lambda^{2-j}T^*M\otimes\ad P$ to $\Lambda^2 T^*M$, and since $k>\frac{1}{2}\dim M+1$, it follows that $\langle\,\wedge\,\rangle$ is a continuous map from $\Omega_k^j(M,\ad P)\times\Omega^{2-j}(M,\ad P)$ to $\Omega_k^2(M)$ \cite[Theorem 9.10]{Palais68}, which in turn is embedded in $\tilde\Omega^2(M)$.
		\item The continuity of $\d:\tilde\Omega^2(M)\to\Omega_{C^0}^3(M)$, where $\Omega_{C^0}^3$ denotes the space of continuous $3$-forms with the $C^0$ norm, implies the closedness of the kernel $\tilde{Z}^2(M)$. If $\Sigma\hookrightarrow M$ is a smoothly immersed surface, then the integration operator $\int_\Sigma:\tilde{Z}^2(M)\to\R$ is continuous and its kernel $K_\Sigma$ is closed in $\tilde{Z}^2(M)$, and hence in $\tilde\Omega^2(M)$. Since every integral second homology class is represented by some $\Sigma$ \cite{Thom54}, we conclude that $\tilde{Z}^2(M) = \cap_\Sigma K_\Sigma$ is closed.
		\item This follows from i.\ and ii., and from the fact that a continuous multilinear map of Banach spaces is smooth \cite{Lang62}.
	\end{enumerate}
\end{proof}

\begin{definition}\label{def:symplectic_form_on_space_of_connections}
	For each $A\in\A_k$, define
	\[
		\omega_A:T_A\A_k\times T_A\A_k \to \tilde\Omega^2(M)/\tilde{B}^2(M)
	\]
	by
	\[
		\omega_A(\alpha_A,\beta_A) = \langle\alpha\wedge\beta\rangle_{\tilde\Omega^2/\tilde{B}^2},
	\]
	where $\alpha_A\in T_A\A_k$ is the vector canonically associated to $\alpha\in\Omega_k^2(M,\ad P)$ by the affine action $\Omega_k^1(M,\ad P)\curvearrowright\A_k$.
\end{definition}

\begin{theorem}\label{thm:canonical_symplectic_structure_on_general_space_of_connections}
	If $\dim M\geq 3$, then $\omega$ is a smooth $\tilde\Omega^2(M)/\tilde{B}^2(M)$-symplectic structure on $\A_k$.
\end{theorem}

\begin{proof}
	For each $A\in\A_k$, Lemma \ref{lem:smooth_induced_product} implies that $\omega_A$ is a continuous bilinear form on $T_A\A_k$ with coefficients in $\tilde\Omega^2(M)/\tilde{B}^2(M)$. The smoothness and closedness of $\omega$, as an assignment of bilinear forms to the tangent fibers of $\A_k$, follow from the fact that $\omega$ corresponds to the function with constant value $\langle\;\wedge\;\rangle_{\tilde\Omega^2/\tilde{B}^2}$ under the natural identifications. It remains to show that $\omega$ is nondegenerate. Fix $A\in\A_k$ and assume that $\alpha\in\Omega_k^1(M,\ad P)\cong T_A\A_k$ is nonzero at $x\in M$. Let $N\subseteq M$ be a closed ball containing $x$, so that $N$ is a submanifold with boundary and
	\[
		\Omega_k^1(N,\ad P) \cong \Omega_k^1(N,\g) \cong \underbrace{\Omega_k^1(N)\times\cdots\times\Omega_k^1(N)}_{\dim \g}.
	\]
	As $\alpha$ is nonzero on $N$, Proposition \ref{prop:infinite_dimensional_vector_valued_symplectic_vector_space_toy_model} provides a $\beta\in\Omega_k^1(N,\ad P)$ with
	\[
		\langle\alpha|_N\wedge\beta\rangle \notin \tilde{B}^2(N).
	\]
	Since $N$ is closed, we may smoothly extend $\beta$ to $M$, in which case
	\[
		\langle\alpha\wedge\beta\rangle \notin \tilde{B}^2(M),
	\]
	and we conclude that $\omega$ is nondegenerate.
\end{proof}

Given a connection $A\in\A$, the exterior covariant derivative $\d_A:\Omega^\ell(P,\g)\to\Omega^{\ell+1}(P,\g)$ is defined by
\[
	\d_A \sigma(X_1,\ldots,X_{\ell+1}) = \d\sigma(h_A X_1,\ldots,h_A X_{\ell+1}),	\hspace{.8cm}X_i\in T_uP,
\]
where $h_A:TP\to A$ is the fiberwise horizontal projection induced by the splitting $A\oplus VP$, for $VP$ the vertical tangent bundle of $P$. Since $\d_A$ preserves the subspace of tensorial forms of type $\ad\,G$, we also consider the exterior covariant derivative as a map $\d_A:\Omega^\ell(M,\ad P)\to\Omega^{\ell+1}(M,\ad P)$. The operations $\langle\;\wedge\;\rangle$ and $\d_A$ are related as follows \cite{AtiyahBott83},
\[
	\d\langle\alpha\wedge\beta\rangle = \langle\d_A\alpha\wedge\beta\rangle + (-1)^{\mathrm{deg}\,\alpha}\langle\alpha\wedge\d_A\beta\rangle.
\]

Since the Lie bracket $[\,,]_\g$ is $\Ad$-equivariant, there is an induced map $[\,,]$ on the fibers of $\ad P$, and we define the composition
\[
	[\;\wedge\;]:\Omega^1(M,\ad P)\times\Omega^1(M,\ad P) \xrightarrow{\wedge_{\Omega^1(M)}} \Omega^2(M,\ad P\otimes\ad P)
		\xrightarrow{[\,,\,]_{\ad P} } \Omega^2(M).
\]
By replacing $\langle\,,\rangle$ with $[\,,]$, the proof of Lemma \ref{lem:smooth_induced_product} shows that $[\;\wedge\;]$ is a continuous bilinear map from $\Omega_k^1(M,\ad P)\times\Omega_k^1(M,\ad P)$ to $\tilde\Omega^2(M)/\tilde{B}^2(M)$.

\begin{lemma}\label{lem:tangent_to_0}
	The map
	\[
		\phi:\Omega_k^1(M,\ad P)\to\Hom\big(\gcal_{k+1},\tilde\Omega^2(M)/\tilde{B}^2(M)\big)
	\]
	given by
	\[
		\phi(\alpha)(f)=\frac{1}{2}\big\langle[\alpha\wedge\alpha]\wedge f\big\rangle_{\tilde\Omega^2/\tilde{B}^2},	\hspace{.6cm}f\in\Omega_{k+1}^0(M,\ad P)\cong\gcal_{k+1},
	\]
	is tangent to $0$, that is, vanishes to first order at $0\in\Omega_k^1(M,\ad P)$.
\end{lemma}

\begin{proof}
	Since $[\;\wedge\;]$ and $\langle\;\wedge\;\rangle_{\tilde\Omega^2/\tilde{B}^2}$ are continuous and bilinear, it follows that the function
	\[
		\Phi:\Omega_k^1(M,\ad P)\times \Omega_k^1(M,\ad P)\to\Hom\big(\gcal_{k+1},\tilde\Omega^2(M)/\tilde{B}^2(M)\big),
	\]
	defined by
	\[
		\Phi(\alpha,\beta)(f) = \frac{1}{2}\big\langle[\alpha\wedge\beta]\wedge f\big\rangle_{\tilde\Omega^2/\tilde{B}^2},
	\]
	is continuous and bilinear. Consequently,
	\[
		\big\|\Phi(\alpha,\beta)\big\|_{k,2}\leq K\|\alpha\|_{k,2}\,\|\beta\|_{k,2}, \hspace{.8cm}\alpha,\beta\in\Omega_k^1(M,\ad P),
	\]
	for some $K>0$ \cite{Lang85}, so that $\|\phi(\alpha)\|\leq K\|\alpha\|^2$ and $\phi$ is tangent to $0$.
\end{proof}

\begin{lemma}\label{lem:smooth_moment_candidate}
	The function
	\[
		\mu: \A_k \to	 \Hom\big(\gcal_{k+1},\tilde\Omega^2(M)/\tilde{B}^2(M)\big),
	\]
	given by
	\[
		\mu(A)(f) = \langle F_A\wedge f\rangle_{\tilde\Omega^2/\tilde{B}^2},	\hspace{.6cm}f\in\Omega_{k+1}^0(M,\ad P)\cong\gcal_{k+1},
	\]
	is smooth. In particular, $\d\mu_A(\alpha)(f) = \langle\d_A\alpha\wedge f\rangle_{\tilde\Omega^2/\tilde{B}^2}$ for $\alpha\in \Omega_k^1(M,\ad P)\cong T_A\A_k$.
\end{lemma}

\begin{proof}
	Fix $A\in\A_k$. As $\langle\;\wedge\;\rangle_{\tilde\Omega^2/\tilde{B}^2}$ and $\G_{k+1}\curvearrowright\A_k$ are smooth, it follows that
	\[
		\psi_A :\Omega_k^1(M,\ad P) \to \Hom\big(\gcal_{k+1},\tilde\Omega^2(M)/\tilde{B}^2(M)\big),
	\]
	where
	\[
		\psi_A(\alpha)(f) = \langle \alpha\wedge \underline{f}_A\rangle_{\tilde\Omega^2/\tilde{B}^2},
	\]
	is smooth. Since $\underline f_A = \d_A f$ and $\d \langle\alpha\wedge f\rangle = \langle\d_A\alpha \wedge f\rangle - \langle\alpha\wedge \d_A f\rangle$, we have
	\[
		\psi_A(\alpha)(f) = \langle\d_A\alpha, f\rangle_{\tilde\Omega^2/\tilde{B}^2}.
	\]
	In conjunction with the identity $F_{A+\alpha} = F_A + \d_A\alpha + \frac{1}{2}[\alpha\wedge\alpha]$ \cite[Lemma 4.5]{AtiyahBott83}, this yields
	\[
		\mu(A+\alpha) = \mu(A) + \psi_A(\alpha) + \phi(\alpha).
	\]
	Since Lemma \ref{lem:tangent_to_0} asserts that $\phi$ is tangent to $0$, and since $\psi_A$ is a continuous linear map, we conclude that $\mu$ is smooth at $A$, with derivative $\d\mu_A=\psi_A$.
\end{proof}

We now present our main result.

\begin{theorem}\label{thm:reduction_of_the_space_of_connections}
	Let $M$ be a compact manifold of dimension at least $3$, $G$ a compact matrix Lie group, $P$ a $G$-principal bundle on $M$ with connected gauge group $\G$, $\A$ the space of connections on $P$, and $k>\frac{1}{2}\dim M+1$ a fixed integer. Denote the the $W^{k,2}$ Sobolev completion of $\A$ by $\A_k$, and likewise for $\G$, $\gcal$, and $\Omega^*$, and write $\tilde\Omega^2(M)$ and $\tilde{B}^2(M)$ for the spaces of $C^1$ forms and coboundaries on $M$, respectively. The function
	\[
		\mu: \A_k \to	 \Hom\big(\gcal_{k+1},\tilde\Omega^2(M)/\tilde{B}^2(M)\big),
	\]
	given by
	\[
		\mu(A)(f) = \langle F_A\wedge f\rangle_{\tilde\Omega^2/\tilde{B}^2},	\hspace{.6cm} f\in\Omega_{k+1}^0(M,\ad P)\cong\gcal_{k+1},
	\]
	is a moment map for the action of $\G_{k+1}$ on $\A_k$ with respect to the polysymplectic structure $\omega\in\Omega^2\big(\A,\tilde\Omega^2(M)/\tilde{B}^2(M)\big)$ defined by
	\[
		\omega(\alpha,\beta) = \langle\alpha\wedge\beta\rangle_{\tilde\Omega^2/\tilde{B}^2},\hspace{.8cm}	\alpha,\beta\in \Omega_k^1(M,\ad P)\cong T_A\A_k.
	\]
	The reduced space $(\A_k)_0$ coincides with the moduli space of flat connections $\M_k=F^{-1}(0)/\G_{k+1}$ on $P$. On the smooth points of $\M_k$, the reduced $2$-form $\omega_0$ takes values in the second cohomology $H^2(M)$.
\end{theorem}

\begin{proof}
	Fix $A\in\A_k$, $\alpha \in \Omega_k^1(M,\ad P)\cong T_A\A_k$ and $f\in \Omega_{k+1}^0(M,\ad P)\cong\bar\gcal_{k+1}$. The identities $\underline f_A = \d_A f$ and $\d \langle\alpha\wedge f\rangle = \langle\d_A\alpha \wedge f\rangle - \langle\alpha\wedge \d_A f\rangle$ yield
	\[
		\langle\d_A\alpha \wedge f\rangle_{\tilde\Omega^2/\tilde{B}^2} = \langle\alpha\wedge\underline{f}_A\rangle_{\tilde\Omega^2/\tilde{B}^2}.
	\]
	Applying Lemma \ref{lem:smooth_moment_candidate} to the left-hand side, and the definition of $\omega$ on the right-hand side, we obtain
	\[
		\d\mu_A(\alpha)(f) = \omega_A\big(\alpha,\underline{f}_A\big).
	\]
	Therefore, $\mu$ is a moment map. For any $A\in\A_k$,
	\[
		\mu(A)=0 \iff \;\big\langle F_A\,\wedge\,\Omega_{k+1}^0(M,\ad P)\big\rangle\subseteq\tilde{B}^2(M)\;  \iff F_A = 0,
	\]
	and we conclude that $\mu^{-1}(0)/\G_{k+1} = F^{-1}(0)/\G_{k+1}$.
	
	If $A\in\mu^{-1}(0)$ and $\alpha,\beta\in\Omega_k^1(M,\ad P)\cong T_A\A_k$ are tangent to $\mu^{-1}(0)$, then $\d_A\alpha = 0$ and $\langle\alpha\wedge\beta\rangle\in \tilde{B}^2(M)$. Consequently, the reduced form $\omega_0$ on $\M_k$ takes values in $\tilde{Z}^2(M)/\tilde{B}^2(M)= H^2(M)$.
\end{proof}

\begin{remark}
	Since Theorem \ref{thm:reduction_of_the_space_of_connections} admits arbitrarily large values of $k>\frac{1}{2}\dim M+1$, and since $\A_k$ and $\G_{k+1}$ are embedded in the H\"older spaces $\A_{C^\ell}$ and $\G_{C^{\ell+1}}$ ($\ell<k+\frac{1}{2}\dim M$), we expect an analogue of Theorem \ref{thm:reduction_of_the_space_of_connections} in the setting of the Fr\'echet manifolds $\A=\cap_\ell\A_{C^\ell}$ and $\G=\cap_\ell\G_{C^\ell}$.
\end{remark}

\begin{corollary}
	If $H^2(M)=0$, then the space of flat connections on $P$ is a Lagrangian submanifold of $\A_k$.
\end{corollary}

\begin{proof}
	If the second cohomology $H^2(M)$ vanishes, then $\omega_0\in\Omega^2\big(\M_k,H^2(M)\big)$ is necessarily zero, and Proposition \ref{prop:Lagrangian_reduction} implies that $\mu^{-1}(0)$ is a Lagrangian submanifold of $\A_k$. The result follows as $F^{-1}(0)=\mu^{-1}(0)$.
\end{proof}

\begin{remark}
	There is a linear multisymplectic form $\omega$ of degree $n$ on the cohomology $\Omega^1(M)$ of an $n$-dimensional manifold $M$ given by
	\[
		\omega(\alpha_1,\ldots,\alpha_n) = \int_M \alpha_1\wedge\cdots\wedge\alpha_n,
	\]
	where the wedge of $n$ forms $\alpha_1,\ldots,\alpha_n\in\Omega^1(M,\ad P)$ is appropriately defined. Given a $G$-principal bundle $P$ as above, we obtain a multisymplectic form of degree $n$ on the space $\A$ of connections on $P$. This form was introduced in \cite{CalliesFregierRogers16}.
\end{remark}


\subsection{Characteristic Forms of Degree $2$ and Ricci Curvature}\label{subsec:characteristic_forms_and_Ricci_curvature}

Let $M$ be a compact manifold with $\dim M\geq 3$, $G$ a compact semisimple matrix Lie group, and $P$ a $G$-principal bundle on $M$. In this section, we apply the polysymplectic reduction procedure to a degenerate $\tilde\Omega^2(M)/\tilde{B}^2(M)$-valued $2$-form on the space of connections $\A_k$.

First we recall a result.

\begin{lemma}[$\!\!$\cite{KobayashiNomizu96}, Lemma II.5.5]\label{lem:vertical_derivative}
	Let $A\in\A$ be a connection, let $\eta\in\Omega^1(P,\g)$ be the connection $1$-form for $A$, and let $\alpha\in\Omega^1(P,\g)$ be a tensorial $1$-form of type $\ad\,G$. Then
	\[
		\d_A\alpha(X,Y) = \d\alpha(X,Y) + \frac{1}{2}\big[\alpha(X),\eta(Y)\big]_\g + \frac{1}{2}\big[\eta(X),\alpha(Y)\big]_\g,
	\]
	for $X,Y\in T_uP$, $u\in P$.
\end{lemma}

We leverage this fact to establish the following.

\begin{lemma}\label{lem:modified_connection_reduction_technical_lemma}
	Suppose that $\alpha\in\Omega^1(M,\ad P)$. If $\phi\in\g^*$ is invariant under the coadjoint action of $G$, then
	\begin{enumerate}[i.]
		\item the assignment
			\begin{align*}
							\;\ad P\;			&\longrightarrow			\R			\\
							\big[(u,\xi)\big]	&\longmapsto				\phi(\xi)
			\end{align*}
			is well-defined and fiberwise linear. We denote this assignment, as well as the induced maps $\Omega^\ell(M,\ad P)\to\Omega^\ell(M)$, by $\phi$.
		\item $\phi(\d_A\alpha) = \d(\phi\alpha)$
		\item $\d(\phi\alpha\wedge\phi\beta) = \phi\,\d_A\alpha\wedge\phi\beta - \,\phi\alpha\wedge\phi\,\d_A\beta$
	\end{enumerate}
\end{lemma}

\begin{proof}
	\begin{enumerate}[i.]
		\item Suppose $(u,\xi)$ and $(u',\xi')\in P\times\g$ represent the same element in $\ad P=P\times_\Ad\g$, that is, there is a $g\in G$ with $u'=ug^{-1}$ and $\xi'=\Ad_g \xi$. Since $\phi$ is $\Ad^*$-invariant, we have $\phi(\xi)=\phi(\xi')$ and thus $\bar\phi$ is well defined.
		\item Let $\xi,\xi'\in\g$ be arbitrary and observe that
			\[
				\phi [\xi,\xi']_\g = \frac{\d}{\d t}\,\phi\:\Ad_{\exp(t\xi)} \xi'\big|_{t=0} = 0.
			\]
			Thus, by Lemma \ref{lem:vertical_derivative} and the linearity of $\phi$, we obtain
			\[
				\phi(\d_A\alpha) = \phi(\d\alpha) + \frac{1}{2}\phi[\alpha,\eta]_\g + \frac{1}{2}\phi[\eta,\alpha]_\g = \d(\phi\,\alpha),
			\]
			as required.
		\item By part (ii), we have
			\begin{align*}
				\d(\phi\alpha\wedge\phi\beta)
					&=	\d\phi\alpha\wedge\phi\beta - \phi\alpha\wedge\d\phi\beta					\\
					&=	\phi\,\d_A\alpha\wedge\phi\beta - \,\phi\alpha\wedge\phi\,\d_A\beta.
			\end{align*}
	\end{enumerate}	
\end{proof}

\begin{remark}
	For any connection $A\in\A$, the image of the curvature $F_A$ under the induced map $\phi:\Omega^2(M,\ad P)\to\Omega^2(M)$ represents in $H^2(M)$ the characteristic class corresponding to the $\Ad$-invariant $1$-linear map $\phi:\g\to\R$. That is, $[\phi F_A]_{H^2}$ is the image of $\phi:\g\to\R$ under the Chern-Weil homomorphism. We will call $\phi F_A$ the \emph{characteristic form} of $A$ associated to $\phi$.
\end{remark}

\begin{proposition}
	Suppose $\dim M\geq 3$.
	\begin{enumerate}[i.]
		\item The collection of maps
		\[
			(\omega_\phi)_A:T_A\A_k\times T_A\A_k \to \tilde\Omega^2(M)/\tilde{B}^2(M),		\hspace{.8cm}A\in A_k,
		\]
		given by
		\[
			(\omega_\phi)_A(\alpha,\beta) = (\phi\alpha \wedge \phi\beta)_{\Omega^2/B^2}, 	\hspace{.8cm}\alpha,\beta\in\Omega^1(M,\ad P)\cong T_A\A,
		\]
		defines a smooth closed $\tilde\Omega^2(M)/\tilde{B}^2(M)$-valued $2$-form on $\A_k$.
		\item The function
		\[
			\mu: \A_k \to	 \Hom\big(\gcal_{k+1},\tilde\Omega^2(M)/\tilde{B}^2(M)\big)
		\]
		defined by
		\[
			\mu(A)(f) = (\phi F_A\wedge\phi f)_{\tilde\Omega^2/\tilde{B}^2},	\hspace{.6cm}f\in\Omega_{k+1}^0(M,\ad P)\cong\gcal_{k+1}
		\]
		is smooth.
	\end{enumerate}
\end{proposition}

\begin{proof}
	The proof analogous to those of Lemma \ref{lem:smooth_induced_product}, Theorem \ref{thm:canonical_symplectic_structure_on_general_space_of_connections}, and Lemma \ref{lem:smooth_moment_candidate}.   Specifically, we replace $\langle\;\wedge\;\rangle$ with $\phi\wedge\phi$ in Lemma \ref{lem:smooth_induced_product}, $\langle\;\wedge\;\rangle_{\tilde\Omega^2/\tilde{B}^2}$ with $(\phi\wedge\phi)_{\tilde\Omega^2/\tilde{B}^2}$ in the corresponding part of Theorem \ref{thm:canonical_symplectic_structure_on_general_space_of_connections}, and $\langle F\wedge\;\rangle$ with $\phi F\wedge\phi$ in Lemma \ref{lem:smooth_moment_candidate}.
\end{proof}
	
\begin{theorem}\label{thm:modified_connection_reduction}
	The action of $\G_{k+1}$ on $(\A_k,\omega_\phi)$ is Hamiltonian, the function
	\[
		\mu_\phi : \A_k \to \Hom\big(\gcal_{k+1},\tilde\Omega^2(M)/\tilde{B}^2(M)\big)
	\]
	given by
	\[
		\mu_\phi(A)(f) = (\phi F_A \wedge \phi f)_{\tilde\Omega^2/\tilde{B}^2},		\hspace{1cm} f\in\Omega_{k+1}^0(M,\ad P)\cong\gcal_{k+1},
	\]
	is a moment map, and the reduced space $(\A_k)_0$ is $(\phi F)^{-1}(0)/\G_{k+1}$.
\end{theorem}

\begin{proof}
	Fix a connection $A\in\A_k$. For $f\in\Omega_{k+1}^0(M,\ad P)\cong\gcal_{k+1}$ and $\alpha\in\Omega_k^1(M,\ad P)\cong T_A\A_k$, Lemma \ref{lem:modified_connection_reduction_technical_lemma} implies that
	\begin{align*}
		\d(\mu_\phi)_A(\alpha)(f)
			&=	(\phi\d_A\alpha\wedge \phi f)_{\tilde\Omega^2/\tilde{B}^2}		\\
			&=	(\phi\alpha\wedge \phi\d_A f)_{\tilde\Omega^2/\tilde{B}^2}		\\
			&=	\omega(\alpha,\underline{f}_A),
	\end{align*}
	and thus $\mu_\phi$ is a moment map for the action of $\G_{k+1}$ on $\A_k$. Finally,
	\[
		\mu_\phi(A)=0 \iff \;\forall f\in\Omega_{k+1}^0(M,\ad P)\,:\,\phi F_A\wedge \phi f \in \tilde{B}^2(M)\;  \iff F_A \in\ker\phi,
	\]
	so that $\mu_\phi^{-1}(0)/\G_{k+1} = (\phi F)^{-1}(0)/\G_{k+1}$.
\end{proof}

Consider a complex manifold $M$ and a holomorphic vector bundle $E$ over $M$. Recall that the first Chern class $c_1(E)$ is represented by the form
	\[
		c_1(A)=\frac{-1}{2\pi\i}\:\tr\,F_A \in\Omega^2(M),
	\]
where $\tr$ denotes the complex trace of $F_A\in\Omega^2(M,\End_\C E)$, and where $A$ is any connection on the holomorphic frame bundle $PE$. We will call $c_1(A)$ the \emph{first Chern form} of $A$. If $A$ is the Chern connection of a Hermitian structure $h:E\otimes \overline{E}\to\C$, then $c_1(A)$ is proportional to the Ricci form $\rho(h)$ of $h$ \cite[Chapter IX]{KobayashiNomizu96a}. This motivates the following terminology.

\begin{definition}
	We call the connection $A\in \A(E)$ \emph{Ricci flat} if $c_1(A)=0$.
\end{definition}

\begin{corollary}
	Let $M$ be a compact complex manifold and let $E$ be a holomorphic vector bundle over $M$ with $c_1(E)=0$. In the sense of $W^{k,2}$ Sobolev completions, the moduli space of Ricci flat connections is the polysymplectic reduction of the space of connections $\A_k(E)$ equipped with the polysymplectic form $\omega_{\tr}$ and moment map given by $A\mapsto\tr\,F_A$, where $\tr$ represents the fiberwise complex trace of $F_A\in\Omega_{k-1}^2(M,\End\, TM^\C)$.
\end{corollary}

\begin{proof}
	Apply Theorem \ref{thm:modified_connection_reduction} with $\phi=\tr$.
\end{proof}

\begin{remark}
	Consider the map $f:\mathrm{Met}(E)\to \A(PE)$ from the space of Hermitian structure on $E$ to the space of connections on $PE$, the frame bundle of $E$, which sends a Hermitian structure $h$ to its Chern connection $f(h)$. Then $f$ is equivariant under the action of the gauge group, $f^*\omega_\tr$ is an $\Omega^2(M)/B^2(M)$-valued $2$-form on $\mathrm{Met}(E)$, and the polysymplectic reduction of $\big(\mathrm{Met}(E),f^*\omega_\tr\big)$ with respect to the moment map $f^*\mu_\tr$ is the moduli space of Ricci flat Hermitian structures on $E$.
	
	In the case that $E=TM^\C$ is the complexified tangent bundle, then the reduced space is the moduli space of Ricci flat K\"ahler metrics on $M$.
\end{remark}

\begin{remark}
	It is significant in the preceding material that $\tr$ denotes the complex trace. Indeed, the argument cannot be adapted to Riemannian structures as $\tr_\R F_A=0$ for any metric connection $A$.
\end{remark}

\appendix
\section{The Infinite-Dimensional Setting}

The purpose of this appendix is to supply some general remarks on the infinite-dimensional setting for polysymplectic geometry.

Whereas the notion of a smooth space is unambiguous in finite dimensions, there are multiple inequivalent formalisms in the infinite-dimensional setting. In this paper, we have chosen to follow the traditional route and work with Hilbert and Banach manifolds. However, it is worth noting that in many situations this approach is unavailable or unsuitable \cite{KrieglMichor97}. In this appendix, we consider the modern approach of employing \emph{convenient vector spaces} as local models for infinite-dimensional manifolds. This theory was introduced in \cite{FrolicherKriegl88} and is thoroughly detailed in \cite{KrieglMichor97}. We choose this framework here for its flexibility as well as its potential for future developments, though at the same time we remark that many of the following observations remain valid within other formalisms as well.

Generally speaking, the transition from finite to infinite-dimension geometry is not straightforward, and a naive approach based on formal analogy is liable to errors. For example, the equivalence between the space of derivations of germs of smooth functions and the space of first-order equivalence classes of smooth curves at a point in the definition of the tangent space in finite dimensions \cite[Section I.1]{KobayashiNomizu96a} does not obtain in infinite dimensions \cite[Section 28.1]{KrieglMichor97}. In this case we distinguish between an \emph{operational tangent space} of bounded derivations and a \emph{kinematic tangent space} of derivations induced along smooth curves.

Many well-known properties of symplectic manifolds are not preserved by the transition to infinite dimensions. This is predominantly a consequence of the fact that in infinite dimensions the induced map
\begin{align*}
	\iota\hspace{.5pt}\omega: 	TM	&\to			T^*M					\\
								X	&\mapsto		\iota_X\omega
\end{align*}
is guaranteed only to be an inclusion of vector bundles, whereas in finite dimensions it is an isomorphism. It is interesting to observe that this particular loss of structure is precisely that responsible for the weakening of results in the finite-dimensional polysymplectic setting. Likewise, in both the infinite-dimensional symplectic and the finite-dimensional polysymplectic settings, we find that not every smooth function is Hamiltonian and that the double symplectic orthogonal does not fix subspaces. Insofar as the polysymplectic formalism does not enjoy these properties and their consequences in the finite dimensional context, their absence in infinite dimensions cannot be said to constitute a loss.

Let us briefly review the underlying ideas of the convenient vector space approach to manifold theory. Unlike the theories of Hilbert, Banach, and Fr\'echet spaces, the theory of convenient vector spaces is grounded not in the framework of \emph{topological vector spaces} but in the construction of a \emph{bornology}, that is, a collection of bounded sets. This bornology is used to identify the space of smooth curves $C^\infty(\R,U)$. We say that a locally convex vector space $U$ is a \emph{convenient vector space} if every smooth curve $c\in C^\infty(\R,U)$ possesses an antiderivative $C\in C^\infty(\R,U)$ \cite[Theorem 2.41]{KrieglMichor97}. Here, as regards differentiation, we have in mind the limit of the familiar difference quotient. The spaces of smooth curves also suffice to determine the smoothness of maps. In particular, a map of convenient vector spaces $f:U\to U'$ is \emph{smooth} precisely when it preserves the space of smooth curves. The associated manifold theory proceeds in the natural way, with local transition functions required to be smooth maps of convenient vector spaces. The reader is referred to \cite{KrieglMichor97} for a very readable introduction.

As all the relevant geometric structures from the finite dimensional setting appear in the infinite-dimensional formalism of convenient manifolds, our $V$-symplectic definitions require no modification. Thus let $M$ be a smooth manifold modeled on a convenient vector space $U$, and consider a $V$-symplectic structure $\omega\in\Omega^2(M,V)$, for some convenient vector space $V$. The main result for our purposes is the following.

\begin{theorem}
	The polysymplectic gradient $s\text{-}\mathrm{grad}:C_H^\infty(M,V)\to\X(M)$ and the bracket $\{\,,\}:C_H^\infty(M,V)\times C_H^\infty(M,V)$ are well-defined. Moreover, the bracket $\{\,,\}$ defines a Lie algebra structure on $C_H(M,V)$.
\end{theorem}

\begin{proof}
	The proof of each statement is precisely analogous to that of the corresponding assertion of \cite[Theorem 48.8]{KrieglMichor97}. Note that some of the conventions in their proof differ from our own by a factor of $-1$.
\end{proof}

The importance of this result lies in its provision of those structures with respect to which we may define a comoment map, and, by extension, a $V$-Hamiltonian system. Within the the convenient formalism, the suitable definition of a Lie group $G$ is given in \cite[Chaper VIII]{KrieglMichor97}. While a detailed analysis of the general situation is beyond the scope of this paper, we remark that an argument similar to the proof of \cite[Theorem 49.16 (4),(5)]{KrieglMichor97} implies that a vector-valued Hamiltonian reduction theorem does obtain in the convenient formalism when $G$ is finite dimensional.

\bibliography{Blacker}
\bibliographystyle{abbrv}

\par
\medskip
\begin{tabular}{@{}l@{}}
	\textsc{Department of Mathematics, East China Normal University,} \\
	\textsc{500 Dongchuan Road, Shanghai, 200241 P.R.\ China} \\[1.5pt]
	\textit{E-mail address}: \texttt{cblacker@math.ecnu.edu.cn}
\end{tabular}

\end{document}